\documentclass[10pt]{amsart}
\usepackage[latin1]{inputenc}
\usepackage[T1]{fontenc}
\usepackage{amsmath,amscd}
\usepackage{graphicx}
\usepackage[all]{xy}
\usepackage{amsthm}
\usepackage{amssymb,url}
\usepackage[charter]{mathdesign}
\usepackage{amsfonts}
\usepackage{mathrsfs}
\usepackage{amsxtra}
\usepackage{bbm}
\usepackage{ae}
\usepackage[all]{xy}
\usepackage{color}

\newtheorem{thm}{Theorem}[section]
\newtheorem{thm*}{Theorem}
\newtheorem{lem}[thm]{Lemma}
\newtheorem{cor}[thm]{Corollary}
\newtheorem{prop}[thm]{Proposition}

\theoremstyle{remark}
\newtheorem{remark}[thm]{Remark}
\newtheorem{example}[thm]{Example}

\theoremstyle{definition}
\newtheorem{deef}[thm]{Definition}

\newcommand{\R}{\mathbbm{R}}
\newcommand{\C}{\mathbbm{C}}

\newcommand{\Z}{\mathbbm{Z}}

\newcommand{\N}{\mathbbm{N}}

\newcommand{\ellL}{\mathcal{L}}

\newcommand{\rd}{\mathrm{d}}

\newcommand{\Ko}{\mathcal{K}}

\newcommand{\Bo}{\mathcal{B}}

\newcommand{\He}{\mathcal{H}}
\newcommand{\T}{\mathbbm{T}}

\newcommand{\Dom}{\mathrm{Dom}\,}

\renewcommand{\epsilon}{\varepsilon}

\newcommand{\supp}{\mathrm{s} \mathrm{u} \mathrm{p} \mathrm{p}\,}
\newcommand{\e}{\mathrm{e}}
\newcommand{\tra}{\mathrm{t}\mathrm{r}}

\title[Commutator estimates on contact manifolds]{Commutator estimates on contact manifolds and applications}
\author{Heiko Gimperlein, Magnus Goffeng}
\address{Heiko Gimperlein,\newline
\indent Maxwell Institute for Mathematical Sciences and \newline
\indent Department of Mathematics, Heriot-Watt University\newline
\indent Edinburgh EH14 4AS\newline
\indent United Kingdom\newline
\newline
\indent {\it and}\newline
\newline 
\indent Institute for Mathematics, University of Paderborn\newline
\indent Warburger Str.~100\newline
\indent 33098 Paderborn\newline
\indent Germany\newline
\newline
\indent Magnus Goffeng,\newline
\indent Institut f\"ur Analysis, Leibniz Universit\"at Hannover\newline
\indent Welfengarten 1\newline
\indent 30167 Hannover\newline
\indent Germany\newline}
\email{h.gimperlein@hw.ac.uk, goffeng@math.uni-hannover.de}

\begin{document}
\maketitle

\begin{abstract}
This article studies sharp norm estimates for the commutator of pseudo-differential operators with multiplication operators on closed Heisenberg manifolds. In particular, we obtain a Calder\'{o}n commutator estimate: If $D$ is a first-order operator in the Heisenberg calculus and $f$ is Lipschitz in the Carnot-Carath\'eodory metric, then $[D,f]$ extends to an $L^2$-bounded operator. Using interpolation, it implies sharp weak--Schatten class properties for the commutator between zeroth order operators and H\"older continuous functions. We present applications to sub-Riemannian spectral triples on Heisenberg manifolds as well as to the regularization of a functional studied by Englis-Guo-Zhang.
\end{abstract}

\large
\section{Introduction}
\normalsize

Let $M$ be a closed manifold equipped with a bracket generating hyperplane bundle $H\subseteq TM$. The objects of study in this paper are commutators $[P,f]$, where $P$ is a pseudo-differential operator from the Heisenberg calculus on $M$ associated with $H$ and $f$ is a H\"{o}lder continuous function on $M$ with respect to the sub--Riemannian Carnot-Carath\'eodory metric. The main results are sharp norm estimates for commutators $[P,f]$. Operator norm estimates on $L^2(M)$ are obtained when $P$ is of first order. Weak Schatten norm estimates are obtained when $P$ is of zeroth order. The weak Schatten class norms describe the asymptotic behavior of its singular values. In particular, we obtain a Calder\'{o}n commutator estimate on Heisenberg manifolds.

\subsection{Motivation} Sharp norm estimates of commutators are motivated by applications. In noncommutative geometry, operator norm estimates for commutators of first-order operators with non-smooth functions are vital for the understanding of the Lipschitz algebra and the Connes metric associated with a spectral triple \cite{c}. From a spectral-theoretic perspective, the weak Schatten norm estimate for the commutator of a zeroth order operator with a H\"older continuous function is a first step towards showing spectral asymptotics for Hankel operators with symbols of low regularity. Our estimate sharpens the Schatten norm bound needed for an analytic degree formula for a non-smooth mapping in \cite{goffodd, goffeven}. It allows to study a regularized trace--functional from higher--dimensional complex analysis \cite{engzh}.

Previous work in this direction has been carried out by many authors. For classical pseudo-differential operators, the first-order estimate was first proven by Calder\'{o}n \cite{calderonone} and later generalized to arbitrary pseudo-differential operators in \cite{ceemtwo}. An overview of these results can be found in \cite[Chapter 3.6]{nlin}. Non-sharp versions of the weak Schatten estimate for zeroth order operators have been obtained in special cases by the second author in \cite{goffodd, goffeven}. The proofs there were based on the explicit singularity of the relevant integral kernels and mixed $L^p$-estimates combined with Russo's Theorem \cite{russo}. The greatest generality (to the authors' knowledge) under which weak Schatten norm estimates could be derived can be found in \cite{birsolinterpolation}. The assumptions of \cite{birsolinterpolation} are very mild and fit well with weakly singular kernels and the Sobolev-type scales on Riemannian manifolds. The methods of \cite{birsolinterpolation} are ill-suited for the anisotropic behavior in the Sobolev scale  appearing on sub-Riemannian manifolds. Such Sobolev spaces are far less studied than their Riemannian counterparts, see \cite{muellerbesov}.

Results of this type have also been obtained in the context of Hankel operators. Of particular interest are results that give an exact characterization of the symbols defining Schatten class operators of a given exponent; \cite{peller} treats $S^1$ and \cite{fangxia,fangxia2,fangxia3} treat $S^{2n-1}$. Holomorphic symbols were treated for arbitrary pre-compact strictly pseudo-convex domains in $\C^n$ in \cite{peloso}. Schatten class properties in higher dimensions (cf. \cite{fangxia,peloso}) come with with a "cutoff property"; for certain types of Hankel operators, there are no nonzero operators that belong to a $p$--Schatten class for $p$ smaller than the homogeneous dimension. Similar features have been observed in other settings, see for instance \cite{alroz,groz,rozshi}.

\subsection{Setup of the paper} We let $M$ denote a closed Riemannian Heisenberg manifold, i.e. the manifold $M$ is equipped with a hyperplane bundle $H\subseteq TM$ and a Riemannian metric. The metric and hyperplane bundle are fixed. By an abuse of notation, $M$ will denote the manifold with this choice of data. We always assume that $H$ is bracket-generating; $H+[H,H]=TM$. Here $[H,H]\subseteq TM$ denotes the range of commutators of vector fields horizontal to $H$. The prototypical example of a bracket generating hyperplane bundle comes from a contact structure on a manifold, see Example \ref{contactexample}. 
From the bracket-generating hyperplane bundle $H$ one defines the Carnot-Carath\'eodory metric $\rd_{CC}$ on $M$, which measures the distance between two points along paths tangent to $H$. For details, see Equation \eqref{dcccddeef} below on page \pageref{dcccddeef}. We let $\textnormal{Lip}_{CC}(M)\subseteq C(M)$ denote the Banach algebra of Lipschitz functions with respect to the metric $\rd_{CC}$. Since the mapping $(M,\rd_{CC})\to (M,\rd_{\textnormal{Rie}})$ is Lipschitz, but not bi-Lipschitz, $\textnormal{Lip}(M)$ is strictly contained in $\textnormal{Lip}_{CC}(M)$.

Associated with the sub-bundle $H$ is an algebra $\Psi^*_H(M)$ of Heisenberg type pseudo-differential operators \cite{bg, ponge}. We shall recall the construction in Section \ref{subsectionschatten}. Operators in $\Psi^*_H(M)$ are modelled on sub-Laplacian operators $\Delta_H\in \Psi_H^2(M)$. Ellipticity in the calculus $\Psi^*_H(M)$ will be called $H$-ellipticity; it implies hypo-ellipticity. There is an associated Sobolev scale $W_H^s(M):=(1+\Delta_H)^{-\frac{s}{2}}L^2(M)$. Any operator $D\in \Psi^m_H(M)$ defines a continuous operator from $W^s_H(M)$ to $W^{s-m}_H(M)$ for any $s$ which is closed if $D$ is $H$-elliptic. More generally, for smooth vector bundles $E,E'\to M$ one obtains a space $\Psi^*_H(M;E,E')$ of pseudo-differential operators, and for $f\in \textnormal{Lip}_{CC}(M)$, $f\cdot W^1_H(M;E)\subseteq W^1_H(M;E)$.

\subsection{Main results}

\begin{thm}
\label{mainbdd}
Let $M$ be a closed Riemannian bracket generating Heisenberg manifold and $E, E'\to M$ complex vector bundles thereon. Given any $D\in \Psi^1_H(M;E,E')$ there is a constant $C_{D}>0$ such that for all $f\in \textnormal{Lip}_{CC}(M)$,
\begin{enumerate}
\item the densely defined operator $[D,f]$ extends to a bounded operator from $L^2(M;E)$ to $L^2(M;E')$;
\item the following norm estimate is satisfied:
\[\|[D,f]\|_{\Bo(L^2(M;E),L^2(M;E'))}\leq C_{D}\|f\|_{\textnormal{Lip}_{CC}(M)}.\]
\end{enumerate}
\end{thm}

The estimate is sharp in the sense that for the horizontal exterior differential
$\rd^H\in \Psi^1_H(M;\C,H^*)$ the norm $\|f\|_{C(M)}+\|[\rd^H,f]\|_{\Bo(L^2(M),L^2(M;H^*))}$ is
equivalent to $\|f\|_{\textnormal{Lip}_{CC}(M)}$.

\begin{remark}
The main technical tool in proving Theorem \ref{mainbdd} is the $T(1)$-theorem of \cite{hytmart}. This theorem holds in a greater generality than the situation considered in this paper. A technical generalization of Theorem \ref{mainbdd} to NIS operators might be of interest, see \cite[Definition 5.10]{keonene} or \cite{nspap}. With a view towards applications we restrict ourselves to a setting where a pseudodifferential calculus has been developed.
\end{remark}

We derive an estimate for commutators with zeroth order operators using a lifting argument. More concretely, let $D\in \Psi^1_H(M;E)$ be an $H$-elliptic operator. Assume that $D$ is self-adjoint as an operator on $L^2(M;E)$ when equipped with the domain $W^1_H(M;E)$. The spectral asymptotics for Heisenberg operators (see \cite{ponge}) implies $(D+i)^{-1}\in \mathcal{L}^{d+2,\infty}(L^2(M;E))$, where $d+1=\dim(M)$. Recall the definition of the weak Schatten ideals $\mathcal{L}^{p,\infty}(\He)$ for a Hilbert space $\He$, see for instance \cite{simon}. For a self-adjoint operator $T\in \Psi^0_H(M;E)$ such that $T^2=1$, we can apply Theorem \ref{mainbdd} to a certain self-adjoint matrix of first-order operators involving $T\sqrt{1+\Delta_H}$ and its adjoint, see Equation \eqref{dffromt} on page \pageref{dffromt}. By \cite{sww}, there is a $C>0$ such that for $f\in \textnormal{Lip}_{CC}(M,i\R)$ we obtain an operator inequality involving 
$$ \tilde{T}:=\begin{pmatrix} 0& T\\T&0\end{pmatrix}$$
of the form
\begin{equation}
\label{lipincl}
-C\|f\|_{\textnormal{Lip}_{CC}(M)}(1+\Delta_H)^{-1/2}\leq [ \tilde{T},f]\leq C\|f\|_{\textnormal{Lip}_{CC}(M)}(1+\Delta_H)^{-1/2},
\end{equation}
as self-adjoint operators on $L^2(M;E)$. By combining the operator inequality \eqref{lipincl} with an interpolation argument and some further matrix constructions we arrive at a similar property for arbitrary $T$ and functions H\"older continuous with respect to the Carnot-Carath\'eodory metric. For $\alpha\in (0,1)$, we let $C^\alpha_{CC}(M)$ denote the Banach algebra of H\"older continuous functions with respect to $\rd_{CC}$ of order $\alpha$. We use the convention $C^1_{CC}(M):=\textnormal{Lip}_{CC}(M)$.

\begin{thm}
\label{mainweakschatten}
Let $M$ be a closed Riemannian bracket generating Heisenberg manifold of dimension $d+1$, $E\to M$ a vector bundle thereon, $T\in \Psi^0_H(M;E)$ and $\alpha\in (0,1]$. There is a constant $C_{T,\alpha}>0$ such that any $a\in C^\alpha_{CC}(M)$ satisfies the weak Schatten class norm estimate
\[\|[T,a]\|_{\ellL^{\frac{d+2}{\alpha},\infty}(L^2(M;E))}\leq C_{T,\alpha}\|a\|_{C^\alpha_{CC}(M)}.\]
\end{thm}

For $a \in C^\infty(M)$, it is a well known consequence of the Heisenberg calculus and its Weyl law that $[T,a] \in {\ellL^{d+2,\infty}(L^2(M;E))}$ but generally $[T,a] \notin {\ellL^{d+2-\varepsilon,\infty}(L^2(M;E))}$ for any $\varepsilon>0$.

We can directly deduce weak Schatten class properties for Hankel operators on the boundary of a pre-compact strictly pseudo-convex domain $\Omega$ in a complex manifold of $n$ complex dimensions. The exponent $d+2$ in Theorem \ref{mainweakschatten} is sharp because of the next corollary and \cite[Theorem 1.5 and Theorem 1.6]{fangxia}. We let $P_{\partial \Omega}\in \Psi^0_H(\partial \Omega)$ denote the Szeg\"o projection (see more below in Remark \ref{szegoex}).

\begin{cor}
\label{hankcor}
For $\alpha\in (0,1]$ and $a\in C^\alpha_{CC}(\partial \Omega)$, the Hankel operators $(1-P_{\partial \Omega})aP_{\partial \Omega}$ belong to the weak Schatten class $\mathcal{L}^{2n/\alpha,\infty}(L^2(\partial\Omega))$.
\end{cor}

When $\partial \Omega=S^1$ this result can be found in \cite[Lemma 8.3]{alexpeller} for the H\"older-Zygmund classes. For the unit sphere related results were obtained for small Hankel operators in \cite{liruss}.

The estimates of Theorem \ref{mainbdd} and Theorem \ref{mainweakschatten} are motivated by the applications in Section \ref{applicationssection}. On the one hand, we investigate the Lipschitz algebra of a spectral triple in noncommutative geometry and, in particular, show that the Carnot-Carath\'eodory metric on $M$ is a Gromov--Hausdorff limit of certain sub--Riemannian Connes metrics. Further motivation is given in Subsection \ref{ccmetsubse} and \cite{hasselmann}. The following result appears as Corollary \ref{ghlimit} on page \pageref{ghlimit} in the main body of the text, for precise notations see Subsection \ref{ccmetsubse}.

\begin{thm}
Let $M$ be a closed sub-Riemannian $H$-manifold. There exists a one--parameter family of $H$--elliptic operators $(0,1]\ni \theta\mapsto D_\theta\in \Psi^1_H(M,S^HM )$ such that
$(C^\infty(M), L^2(M, S^HM), D_\theta)$ is a spectral triple of metric dimension $d+2$ such that $(M, d_{D_\theta}) \to (M, d_{CC})$ with respect to Gromov--Hausdorff distance as $\theta \to 0$. 
\end{thm}

On the other hand, we regularize a mulitlinear functional first studied by Englis-Guo-Zhang \cite{engguozh, engzh} on the boundary of a pre-compact strictly pseudo-convex domain $\Omega$ in a complex manifold of $n$ complex dimensions. We consider two families of multilinear functionals $\xi_{k,\omega}$ and $\zeta_{k,\omega}$ on $(C^{n/k}_{CC}(M))^{\hat{\otimes}2k}$, where $\zeta_{n,\omega}$ was first studied by Englis-Guo-Zhang on $C^\infty(\partial\Omega)$. These functionals turn out to have interesting "cutoff-properties" as shown Subsection \ref{subsregfunc}. The following result appears as Theorem \ref{vanishingcor} in the main text, see page \pageref{vanishingcor}. 

\begin{thm}
The multilinear functionals $\xi_{k,\omega}$ and $\zeta_{k,\omega}$ are continuous and in general nonzero but vanish identically on the closed subspace spanned by elements of the form $a_1\otimes \cdots \otimes a_{2k}$ such that for some $j$ and $\beta>n/k$, $a_j\in C^\beta_{CC}(M)$. 
\end{thm}

Even though $\xi_{k,\omega}$ and $\zeta_{k,\omega}$ vanish on smooth functions, they have geometric significance in some cases. Geometric formulas are pursued further in \cite{gg}.

\subsection*{Relations to noncommutative geometry}

Two central objects of study in noncommutative geometry are unbounded and bounded Fredholm modules.
Both objects can be viewed as cocycles for analytic $K$-homology, a cohomology theory for $C^*$-algebras; see more in \cite{baajjulg,blackadar,goffmesland,higsonroe,jeto}. Examples relevant to this paper are given by elliptic operators in the Heisenberg calculus on a bracket generating Heisenberg manifold. It is a direct consequence of the symbol calculus that for a self-adjoint $H$-elliptic Heisenberg operator $D$ of order $1$ acting on a vector bundle $E$, $(C^\infty(M),L^2(M,E),D)$ forms a spectral triple. For a $d+1$-dimensional manifold $M$, it is $(d+2,\infty)$-summable, because of the spectral asymptotics of $H$-elliptic Heisenberg operators \cite{ponge}. The idea behind the proof of Theorem \ref{mainweakschatten} is to lift a bounded Fredholm module to an unbounded Fredholm module via the calculus of Heisenberg operators. Using this lift, we can apply Theorem \ref{mainbdd} through Equation \eqref{lipincl}. Theorem \ref{mainbdd} implies that $\textnormal{Lip}_{CC}(M)$ continuously embedds into the abstractly defined Lipschitz algebra for an unbounded Fredholm module associated with an $H$-elliptic first order operator. Theorem \ref{mainweakschatten} answers the analogous question in the bounded picture of $K$-homology; $C^\alpha_{CC}(M)$ is continuously embedded in the H\"older algebra abstractly defined from a bounded Fredholm module associated with an $H$-elliptic zeroth order operator with the dimensionality being the homogeneous dimension $d+2$ of $M$.

\subsection*{Contents of the paper}

This paper is organized as follows. In Section \ref{sectionone}, we recall some relevant background from sub-Riemannian as well as the Heisenberg pseudodifferential calculus. We also recall the $T(1)$-theorem of \cite{hytmart}; our main technical tool for proving Theorem \ref{mainbdd}. In Section \ref{sectiontwo}, we prove Theorem \ref{mainbdd}. The main idea is to reduce to the case $D=T\cdot X$, where $T$ is of order $0$ and $X$ is a differential operator horizontal to the Heisenberg structure, and apply the $T(1)$-theorem of \cite{hytmart}.

With Theorem \ref{mainbdd} in hand, the Lipschitz case of Theorem \ref{mainweakschatten} follows from Equation \eqref{lipincl} when $T$ is a self-adjoint zeroth order operator with $T^2=1$. In Section \ref{provingweakschattenlipschitz}, we prove the zeroth order estimate of Theorem \ref{mainweakschatten} for Lipschitz functions after reducing to the case $T^2=1$ by means of a matrix construction similar to that of the Bott projection. We prove Theorem \ref{mainweakschatten} in the general case in Section \ref{interpolationsection} by introducing an anisotropic Besov scale in each coordinate chart. This scale is used in an interpolation argument between the spaces of continuous and Lipschitz functions. Section \ref{applicationssection} addresses applications of our estimates in non--commutative geometry and complex analysis.

\subsection*{Notation} 

We use the convention $\N=\{0,1,2,\ldots\}$. The notation $a\lesssim b$ means that there is a constant $C$ such that $a\leq Cb$.  We write $a\sim b$ if $a\lesssim b$ and $b\lesssim a$. Also, we write $C^0_{CC}(M)$ for the space of continuous functions and $C^1_{CC}(M)$ for the space of Lipschitz continuous functions $\textnormal{Lip}_{CC}(M)$ in the Carnot-Caratheodory metric.

\large
\section{Preliminaries}
\label{sectionone}
\normalsize

We give a short introduction to the geometry and analysis of operators appearing in sub-Riemannian geometry, as needed in later sections. The results in this section are well-known. We summarize them here for later reference.

\subsection{Carnot groups}

The local situation of a sub-Riemannian manifold is modeled on a homogeneous group, also called a Carnot group. Let $G$ be a simply connected nilpotent Lie group with Lie algebra $\mathfrak{g}$. The central series
$$V_j:=\sum_{l=1}^j\underbrace{[\mathfrak{g},[\cdots[\mathfrak{g},\mathfrak{g}]\cdots]]}_{l\mbox{  times}}$$
for $\mathfrak{g}$ induces a grading as linear spaces
\[\mathfrak{g}=\bigoplus_{j=1}^R W_j, \quad W_j:=V_j/V_{j-1}.\]
We will write coordinates in $\mathfrak{g}$ as $(x_1,\ldots, x_R)$ where $x_i\in W_i$. Since $G$ is simply connected and nilpotent, $\exp:\mathfrak{g}\to G$ is a diffeomorphism. Recall that a homogeneous group $G$ is a simply connected nilpotent Lie group equipped with an $\R_+$-action $\delta$ such that
\[\delta\!_t( \exp(x_1,\ldots, x_R))=\exp(t^{a_1}x_1,\ldots, t^{a_R}x_R),\]
for an $a=(a_1,\ldots, a_R)\in \R^R$. One example of such a choice of $a$ is $a=(1,2,\ldots, R)$. An element $X\in \mathfrak{g}$ is called horizontal if
\[(\delta_t)_*X=t^{a_1}X.\]

Equipping $W_1,\ldots, W_R$ with a length function naturally induces a length function on $G$ called the Koranyi gauge;
\[|x|_G:=\sqrt[2a_1\cdots a_R]{\sum_{l=1}^R |x_l|^{2a_l\cdots a_R}} .\]
We also define the Koranyi balls
$$B_G(x_0,r):=\{x\in \R^{d+1}: |x_0^{-1}\cdot x|_G<r\}.$$

\subsubsection{Heisenberg group}\label{Heisenbergsection}

A homogeneous group structure on $\R^{d+1}$ can be constructed from an antisymmetric form $L=(L_{jk})_{j,k=1}^d$ on $\R^d$. We denote coordinates by $x=(t,z)\in \R\times \R^d= \R^{d+1}$ and equip $\R^{d+1}$ with the Lie algebra structure defined from
\[[(t,z),(t',z')]=(L(z,z'),0).\]
The corresponding Lie group $\R^{d+1}$, with multiplication given by
\[(t,z)\cdot (t',z')=\left(t+t'+\frac{1}{2}L(z,z'),z+z'\right),\]
is called a Heisenberg group.
We equip $\R^{d+1}$ with the $\R_+$-action defined by
\begin{equation}
\label{actionr}
\lambda.(t,z)=(\lambda^2t,\lambda z).
\end{equation}
This action is indeed an action of $\R_+$ on $\R^{d+1}$ as a Lie group;
\[\lambda.\left((t,z)\cdot (t',z')\right)=\lambda.(t,z)\cdot \lambda.(t',z').\]
Hence $\R^{d+1}$ becomes a homogeneous Lie group in this fashion.

We denote the action of $\R_+$ on functions by $T_\lambda$:  $T_\lambda f(x):=f(\lambda.x)$. A natural length function can be defined on $\R^{d+1}$ from the Koranyi gauge, which for the Heisenberg group takes the form
\[|x|_H=\sqrt[4]{t^2+|z|^4},\]
and is homogeneous for the $\R_+$-action. A basis for the horizontal vector fields on $\R^{d+1}$ is given by the vector fields
\[X_j:=\frac{\partial}{\partial z_j}+\frac{1}{2}\sum_{k=1}^d L_{jk}z_k\frac{\partial}{\partial t}, \quad j=1,2,\cdots, d.\]
The vertical vector field is defined by
\[X_0:=\frac{\partial}{\partial t}.\]
The horizontal vector fields are homogeneous of degree $1$ with respect to the $\R_+$-action and satisfy the commutation relation
\[[X_j,X_k]=L_{jk}X_0.\]

\subsection{Carnot and Heisenberg manifolds}

Assume that $M$ is a smooth manifold. To simplify our discussions we always assume $M$ to be connected and closed. Assume that $\mathcal{V}\subseteq TM$ is a subbundle. We will use the notation
$$\mathcal{V}_{(k)}:=\sum_{j\leq k}\underbrace{[\mathcal{V}[\cdots,[\mathcal{V},\mathcal{V}]\cdots ]]}_{j\mbox{  times}}$$
for iterated commutators of $\mathcal{V}$ with itself. We use the convention $\mathcal{V}_{(0)}=0$ and $\mathcal{V}_{(1)}=\mathcal{V}$.

\begin{deef}
\label{cantotdef}
A Carnot structure of length $R\in \N=\{0,1,2,\ldots\}$ on $M$ is a subbundle $\mathcal{V}\subseteq TM$ such that
\begin{enumerate}
\item $\mathcal{V}_{(R)}=TM$.
\item $\mathcal{V}_{(k)}$ is a subbundle of $TM$ for any $k$.
\end{enumerate}
\end{deef}

A Carnot structure induces a filtration $(\mathcal{V}_{(k)})_{k=1}^R$ of \emph{vector bundles} of $TM$. In the terminology of \cite{gromov}, a Carnot structure is called an equiregular Carnot structure. We let $\mathcal{W}_j:=\mathcal{V}_{(j)}/\mathcal{V}_{(j-1)}$.

\begin{prop}
For any point $x$ on a Riemannian Carnot manifold $M$ defined from a subbundle $\mathcal{V}$, the operation
\begin{align*}
T_xM\times T_xM\ni &(v_1,\ldots,v_R)\times (w_1,\ldots, w_R)\mapsto \\
&(0,[v_1,w_1]\mod \mathcal{V}_{(1)},\ldots, [v_{R-1},w_{R-1}]\mod \mathcal{V}_{(R-1)})\in \bigoplus \mathcal{W}_i=TM
\end{align*}
induces a Lie algebra structure on $T_xM$.
\end{prop}

The proof of this proposition is standard and can be found for instance in \cite{bella96}.

A Heisenberg structure on a manifold $M$ is given by a hyperplane bundle $H\subseteq TM$. We say that the Heisenberg structure is bracket generating if $H+[H,H]=TM$. A bracket-generating Heisenberg structure induces a Carnot structure of length $2$. We will refer to a Riemannian manifold equipped with a bracket-generating Heisenberg structure as a \emph{sub-Riemannian $H$-manifold}. 

\begin{example}
\label{contactexample}
Let $M$ be a contact manifold with contact form $\eta$, i.e. $M$ is  of dimension $2n-1$ and $\eta$ is a one-form such that $\eta\wedge (\rd\eta)^{n-1}$ is non-degenerate. A direct consequence of this condition is that $\rd \eta$ is non-degenerate on $\ker\eta$. For example, if $M = \partial \Omega$ is the boundary of a strictly pseudo-convex domain $\Omega$ in a complex manifold of complex dimension $n$, $\eta = \rd^c\rho$ for a boundary defining function $\rho$. With the contact structure, a Heisenberg structure can be associated by setting $H:=\ker \eta$. By Cartan's formula, for any vector fields $X,Y$,
\[\rd \eta(X,Y)=X(\eta(Y))-Y(\eta(X))-\eta([X,Y]).\]
In particular, because $\rd \eta$ is non-degenerate on $H$, for any $X\in H$ there is a $Y\in H$ which makes the left hand side non-zero, $\eta([X,Y])\neq 0$. Hence, the Heisenberg structure associated with a contact structure is bracket generating. 
\end{example}

\begin{remark}
\label{leviformremark}
On a general Heisenberg manifold, the role of $\rd \eta$ is played by the Levi form which is a section $\mathcal{L}\in C^\infty(M,H^*\wedge H^*)$ whenever $TM/H$ is orientable. The Levi form is the composition of the form 
$$C^\infty(M,H)\times C^\infty(M,H)\ni (X,Y)\mapsto [X,Y]\!\!\!\!\mod H\in C^\infty(M,TM/H),$$ 
with a trivialization $TM/H\cong M\times \R$. Contact manifolds give rise to the extreme case of $\mathcal{L}$ being non-degenerate. The case $\mathcal{L}\equiv 0$ corresponds to a foliation.
\end{remark}

A theorem by Chow \cite{chow} states that if $M$ is a manifold with a bracket generating sub-bundle $\mathcal{V}\subseteq TM$, then given any $x,y\in M$ there is a smooth path $\gamma:[0,1]\to M$ such that $\gamma(0)=x$, $\gamma(1)=y$ and $\dot{\gamma}(t)\in \mathcal{V}_{\gamma(t)}$ for almost all $t\in [0,1]$. As a consequence, after a choice of Riemannian metric on $\mathcal{V}$, a metric known as the Carnot-Carath\'eodory metric can be defined as follows. For any $x,y,\in M$:
\small
\begin{equation}
\label{dcccddeef}
\rd_{CC}(x,y):=\inf\left\{\left(\int_0^1\|\dot{\gamma}(t)\|_\mathcal{V}^2\rd t\right)^{1/2}:\;\gamma(0)=x,\, \gamma(1)=y\;\mbox{and}\;\dot{\gamma}(t)\in \mathcal{V}_{\gamma(t)}\mbox{  a.e.}\right\}.
\end{equation}
\normalsize
For more details see, e.g. Chapter $1.6$ of \cite{montgomery}.

We now restrict our attention to the case of bracket-generating Heisenberg manifolds. We let $M$ denote a $d+1$-dimensional sub-Riemannian $H$-manifold. A Heisenberg chart is a local chart on $M$ and a choice of $H$-frame over the chart, i.e. a local frame $X_0,X_1,\ldots, X_d$ such that $X_1,\ldots, X_d$ span $H$. Recall from \cite[Chapter 2.1]{ponge}, near any point $u\in M$ there is a particular Heisenberg chart where the coordinate systems are privileged in the sense that as $t\to 0$:
\[\delta_t^* X_i=
\begin{cases}
t^{-1} X_i+O(1), \;& i=1,\ldots, d\\
t^{-2} X_0+O(t^{-1}), \;& i=0,
\end{cases}\]
cf.~\cite[Definition 2.1.10]{ponge}. Such a choice of coordinates is also called $y$-coordinates in the literature, \cite[Chapter 3, \S 11]{bg}. In privileged coordinates, the Carnot-Carath\'eodory metric can be related to the Koranyi gauge, cf. \cite{nswpap} and discussion in \cite[Chapter 2.2]{cadapaaaow}:

\begin{prop}
\label{dccvskorga}
Let $\epsilon_u:U\to \R^{d+1}$ be the coordinates in a priviligied chart centered at $u\in M$ depending smoothly on $u\in U$. There is a neighborhood $V$ of $u$ and a constant $C>0$ such that
\[\frac{1}{C} |\epsilon_u(y)|_H\leq \mathrm{d}_{CC}(u,y)\leq C |\epsilon_u(y)|_H,\quad\forall y,u\in V.\]
\end{prop}

It will be useful to describe the Lipschitz and H\"older spaces for the Carnot-Carath\'eodory metric in terms of the Koranyi gauge. We cover $M$ with Heisenberg charts $(U_j)_{j=1}^N$ such that $(\delta_{1/2}^*U_j)_{j=1}^N$ is a collection of open sets covering $M$. Choose a partition of unity $(\chi_j)_{j=1}^N$ subordinate to $(\delta_{1/2}^*U_j)_{j=1}^N$ with $\chi_j$ compactly supported in $U_j$. For small enough $c>0$ we can define the following functions on $M\times M$:
\begin{align*}
\ell_{i,j}^K(x,y)&:=\chi_i(x)\chi_j(y) \cdot\min(|\epsilon_i(x)^{-1}\epsilon_i(y)|_H,c)\quad \mbox{and}\\
\ell_{i,j}^{an}(x,y)&:=\chi_i(x)\chi_j(y)\cdot \min(|\epsilon_i(x)-\epsilon_i(y)|_H,c).
\end{align*}
The indices $K$ and $an$ stand for Koranyi and anisotropic, respectively. We define the quasi-metrics:
\begin{align*}
\rd_K(x,y)&:=\sum_{i,j=1}^N\ell_{i,j}^K(x,y)+\ell_{j,i}^K(y,x)\quad \mbox{and}\\
\rd_{an}(x,y)&:=\sum_{i,j=1}^N\ell_{i,j}^{an}(x,y)+\ell_{j,i}^{an}(y,x).
\end{align*}

\begin{lem}
\label{equivalencesofmetrics}
Let $M$ be a closed sub-Riemannian $H$-manifold. The Carnot-Carath\'eodory metric is equivalent to the anisotropic quasi-metric $\rd_{an}$ and the Koranyi quasi-metric $\rd_K$.
\end{lem}

\begin{proof}
It suffices to prove the assertion in a coordinate neighborhood. We fix $x\in M$. The function $y\mapsto \rd_{CC}(x,y)/|x-y|_H$ is well defined for $y\neq x$ and homogeneous of degree $0$ with respect to the usual action. Similarly, $y\mapsto |x-y|_H/\rd_{CC}(x,y)$ is well defined for $y\neq x$ and homogeneous of degree $0$. Similarly, $|x-y|_H/|xy^{-1}|_H$ and $|xy^{-1}|_H/|x-y|_H$ are homogeneous of degree $0$. A compactness argument, using suitable spheres, guarantees boundedness of these functions. Since $M$ is compact, the assertion follows.
\end{proof}

\subsection{Function spaces}

\begin{deef}[CC-H\"older and CC-Lipschitz continuous functions]
We define the space of CC-Lipschitz functions $\textnormal{Lip}_{CC}(M)$ as the linear space of all functions $f\in L^\infty(M)$ satisfying
\[|f|_{\textnormal{Lip}_{CC}(M)}:=\sup_{x\neq y}\frac{|f(x)-f(y)|}{\rd_{CC}(x,y)}<\infty.\]
Similarly, the space of CC-H\"older continuous functions $C^\alpha_{CC}(M)$ of exponent $\alpha\in (0,1)$ consists of all $f\in L^\infty(M)$ such that
\[|f|_{C^\alpha_{CC}(M)}:=\sup_{x\neq y}\frac{|f(x)-f(y)|}{\rd_{CC}(x,y)^\alpha}<\infty.\]
\end{deef}

The obvious examples of elements of $\textnormal{Lip}_{CC}(M)$ are of course elements of $C^1(M)$, and the inclusion $C^\alpha(M)\subseteq C^\alpha_{CC}(M)$ is continuous. A standard computation shows that $\textnormal{Lip}_{CC}(M)$ is a Banach algebra in the norm
\[\|f\|_{\textnormal{Lip}_{CC}(M)}:=\|f\|_{L^\infty(M)}+|f|_{\textnormal{Lip}_{CC}(M)}\ .\]
Similarly, $C^\alpha_{CC}(M)$ is a Banach algebra in the norm
\[\|f\|_{C^\alpha_{CC}(M)}:=\|f\|_{L^\infty(M)}+|f|_{C^\alpha_{CC}(M)}\ .\]

Recall that we write $C^1_{CC}(M)=\textnormal{Lip}_{CC}(M)$.

\begin{lem}
Let $M$ be a closed Riemannian manifold with a Carnot structure defined from the subbundle $\mathcal{V}\subseteq TM$. For a function $f\in L^\infty(M)$ the following are equivalent:
\begin{enumerate}
\item $f\in \textnormal{Lip}_{CC}(M)$.
\item There is a constant $C>0$ such that for any point $x_0\in M$ and in any ON-frame $(X_j)_{j=1}^{d}$ of $\mathcal{V}$ around the point $x_0$, $\|X_jf\|_{L^\infty}\leq C$ for $j=1,\ldots, d$.
\end{enumerate}
Furthermore, the semi-norm $|\cdot|_{\textnormal{Lip}_{CC}(M)}$ is equivalent to the semi-norm
\[|f|'_{\textnormal{Lip}_{CC}(M)}:=\sup \|X_jf\|_{L^\infty}\]
where the supremum is taken over all points $x_0\in M$, all ON-frames $(X_j)_{j=1}^{d}$ of $\mathcal{V}$ around $x_0$ and all $j=1,\ldots, d$.
\end{lem}

\begin{proof}
Assume that $x,y\in M$ and that $|f|'_{\textnormal{Lip}_{CC}(M)}<\infty$. For any horizontal curve $\gamma:[0,1]\to M$ with endpoints $x$ and $y$,
\[|f(x)-f(y)|\leq \int_0^1 |\rd f(\gamma(t)).\dot{\gamma}(t)|\rd t\leq (\dim \mathcal{V})^{1/2}|f|'_{\textnormal{Lip}_{CC}(M)}\int_0^1 \|\dot{\gamma}(t)\|_\mathcal{V}\rd t.\]
Hence, $|f|_{\textnormal{Lip}_{CC}(M)}\leq (\dim \mathcal{V})^{1/2} |f|'_{\textnormal{Lip}_{CC}(M)}$.

Conversely, assume $f\in \textnormal{Lip}_{CC}(M)$. By restricting to horizontal curves, it can be shown that $Xf\in L^\infty_{loc}$ for any horizontal vector field $X$. Take a point $x_0\in M$ and an $ON$-frame $(X_j)_{j=1}^{d}$ of $\mathcal{V}$ around $x_0$. Let $\gamma:[0,1]\to M$ be a horizontal curve such that $\gamma(0)=x_0$ and $\dot{\gamma}(0)=\sum_{j=1}^d X_j(f)X_j$. Let $L(\epsilon):=\int_0^\epsilon \|\dot{\gamma}(t)\|^2_\mathcal{V}\rd t$. In local coordinates one verifies the identity
\[\lim_{\epsilon\to 0}\frac{|f(\gamma(\epsilon))-f(x_0)|}{L(\epsilon)}=\left\|\sum_{j=1}^dX_j(f)(x_0)X_j(x_0)\right\|_\mathcal{V}.\]
As a consequence,
$$\left\|\sum_{j=1}^dX_j(f)(x_0)X_j(x_0)\right\|_\mathcal{V}\leq |f|_{\textnormal{Lip}_{CC}(M)},$$
or $|f|_{\textnormal{Lip}_{CC}(M)}\geq (\dim \mathcal{V})^{-1/2} |f|'_{\textnormal{Lip}_{CC}(M)}$.
\end{proof}

\begin{prop}
For $s\in (0,1]$, the semi-norm $|\cdot|_{C^s_{CC}(M)}$ is equivalent to both of the following semi-norms:
\[|f|_{C^\alpha_{CC}(M),K}:=\sup\left\{\frac{|f(x)-f(y)|}{|xy^{-1}|_H^\alpha}\bigg|\; x\neq y, \; x,y \in U\right\},\]
\[|f|_{C^\alpha_{CC}(M),an}:=\sup\left\{\frac{|f(x)-f(y)|}{|x-y|_H^\alpha}\bigg|\; x\neq y, \; x,y \in U\right\}.\]
\end{prop}

This proposition is a direct consequence of Lemma \ref{equivalencesofmetrics}.

\subsection{The Heisenberg calculus}
\label{subsectionschatten}

We briefly recall some relevant facts regarding the Heisenberg calculus. For a more detailed account, we refer to \cite{bg,ponge,taylorncom}. Closely related calculi are considered in \cite{christ, melin}. We follow the notations of \cite{ponge}. As above, let $M$ denote a $d+1$-dimensional closed sub-Riemannian $H$-manifold and $U$ a Heisenberg coordinate system on $M$ with Heisenberg frame $X_0,\ldots, X_d$. For a multi-index $\alpha=(\alpha_0,\alpha_1,\ldots, \alpha_d)\in \N^{d+1}$ we set $\langle \alpha\rangle:=2\alpha_0+\sum _{i=1}^d \alpha_i$. Also, recall the $\R_+$-action on $\R^{d+1}$ from Equation \eqref{actionr}, page \pageref{actionr}. The class of symbols relevant for the Heisenberg calculus is defined as follows.

\begin{deef}[Definition $3.1.4$ and $3.1.5$ of \cite{ponge}]
For $m\in \C$, the space $S_m(U\times \R^{d+1})$ of symbols homogeneous of degree $m$ is defined as the space of symbols $p\in C^\infty(U\times \R^{d+1}\setminus \{0\})$ that are homogeneous of degree $m$:
\[p(x,\lambda\cdot \xi)=\lambda^mp(x,\xi)\quad\forall\lambda>0,\;(x,\xi)\in U\times \R^{d+1}.\]
The space $S^m(U\times \R^{d+1})$ of symbols of degree $m$ is defined as the space of those $p\in C^\infty(U\times \R^{d+1})$ that admit a $1$-step polyhomogeneous asymptotic expansion; i.e.~there exists $p_k\in S_{m-k}(U\times \R^{d+1})$ for $k\in \N$ such that for any $N\in \N$ and $K\subseteq U$ compact
\[\left| \partial_x^\alpha\partial_\xi^\beta\left(p-\sum_{k=0}^Np_k\right)(x,\xi)\right|\leq C_{\alpha,\beta,K,N}|\xi|_H^{\Re(m)-\langle \beta\rangle-N},\quad\forall x\in K, \;|\xi|_H\geq 1,\]
for some constant $C_{\alpha,\beta,K,N}>0$.\\
\end{deef}

The frame $X_0,X_1,\ldots, X_d$ gives rise to functions $\sigma_j(x,\xi):=\sigma(X_j)$ -- their classical symbols as differential operators. We let $\sigma:=(\sigma_0,\sigma_1,\ldots,\sigma_d):U\times\R^{d+1}\to \R^{d+1}$ denote the vector of these symbols. Associated with a symbol $p\in S^m(U\times \R^{d+1})$ there is an operator $P:=p(x,-iX):C^\infty_c(U)\to C^\infty(U)$ defined by
\[Pf(x):=(2\pi)^{-d-1}\int \e^{ix\cdot \xi}p(x,\sigma(x,\xi))\hat{f}(\xi)\rd \xi.\]
A $\Psi_H$DO of order $m\in \C$ on $U$ is an operator of the form $p(x,-iX)+R$ for a $p\in S^m(U\times \R^{d+1})$ and an integral operator $R$ with smooth integral kernel. We write $R \in \Psi_H^{-\infty}(U)=\Psi^{-\infty}(U)$ and $p(x,-iX)+R \in \Psi_H^{m}(U)$. Whenever $P_1$ and $P_2$ are two $\Psi_H$DOs of order $m_1$ and $m_2$, respectively, and at least one of them is properly supported, the composition $P_1P_2:C^\infty_c(U)\to C^\infty(U)$ is again a $\Psi_H$DO of order $m_1+m_2$ whose $1$-step polyhomogeneous asymptotic expansion can be expressed by means of the fiberwise convolution product on the bundle of Lie groups $TU\to U$. This is proven in \cite[Theorem 4.7]{bg}. A symbol $p\in S^m(U\times \R^{d+1})$ is called $H$-elliptic, if there is a $\Psi_H$DO $Q$ of order $-m$ such that $p_0(x,-iX)Q-1$ and $Qp_0(x,-iX)-1$ are both $\Psi_H$DOs of order $-1$.

A consequence of the product formula in \cite[Theorem 4.7]{bg} is the following proposition.

\begin{prop}
\label{smoothcom}
Assume that $P \in \Psi_H^{m}(U)$ and $f\in C^\infty(U)$. Then $[P,f] \in \Psi_H^{m-1}(U)$.
\end{prop}

Since the fiber wise convolution product is non-commutative, the analogous assertion fails when $f \in C^\infty(U)$ is replaced by a general operator $Q \in \Psi_H^0(U)$.

$\Psi_H$DOs can also be characterized by kernel estimates. We define $S'_{\textnormal{reg}}(\R^{d+1})$ as the space of tempered distributions that are regular outside of the origin $0\in \R^{d+1}$, see \cite[Definition 3.1.10]{ponge}. We equip the space $S'_{\textnormal{reg}}(\R^{d+1})$ with the smallest topology making the natural mapping $S'_{\textnormal{reg}}(\R^{d+1})\to S'(\R^{d+1})\oplus C^\infty(\R^{d+1}\setminus\{0\})$ continuous. The following definition was originally due to Beals-Greiner, see \cite[Chapter 3, \S15]{bg}. We follow the notations of \cite{ponge}.

\begin{deef}[Definition $3.1.11$ and $3.1.13$ of \cite{ponge}]
\label{kerneldeff}
The space $\mathcal{K}_m(U\times \R^{d+1})\subseteq C^\infty(U)\hat{\otimes}S'_{\textnormal{reg}}(\R^{d+1})$ of homogeneous distributions of degree $m$ consists of those $K=K(x,z)\in C^\infty(U)\hat{\otimes}S'_{\textnormal{reg}}(\R^{d+1})$ such that
\begin{equation}
\label{homdist}
K(x,\lambda z)=\lambda^mK(x,z)+\lambda^m\log\lambda\sum_{\langle \alpha\rangle=m} c_\alpha(x)z^\alpha,
\end{equation}
for some functions $c_\alpha\in C^\infty(U)$. The space $\mathcal{K}^m(U\times \R^{d+1})$ of kernels of degree $m$ is defined as the space of $K\in \mathcal{D}'(U\times\R^{d+1})$ such that for any $N\in\N$, there are $K_l\in \mathcal{K}_{m-l}(U\times \R^{d+1})$, $l\in \N$, with
\[K-\sum_{l=0}^J K_j\in C^N(U\times \R^{d+1})\]
for a large enough $J$.
\end{deef}

\begin{remark}
Logarithmic terms can appear in \eqref{homdist} only when $m\in \N$.
\end{remark}

\begin{remark}
\label{PsiDOprop}
By \cite[Proposition 3.1.16]{ponge}, there is a one-to-one correspondence between $\Psi_H$DOs modulo smoothing operators and kernels as in Definition \ref{kerneldeff} modulo smooth kernels. Differential operators correspond to kernels supported in $z=0$. More precisely, if $P \in  \Psi_H^{m}(U)$, there is a $K_P\in \mathcal{K}^{-m-d-2}(U\times \R^{d+1})$ and a smoothing $R\in C^\infty(U\times U)$ such that the Schwartz kernel $k_P$ of $P$ can be written as
\[k_P(x,y)=|\epsilon'_x|K_P(x,-\epsilon_x(y))+R(x,y),\]
where $\epsilon_x$ denotes privileged coordinates which depends smoothly on $x$. The homogeneous expansion of the kernel $K_P$ guarantees that for any horizontal vector field $X$, the estimates
\begin{align*}
|K_P(x,-\epsilon_x(y))| &\lesssim \rd_{CC}(x,y)^{-(d+1+m)}\\
|X_{x} K_P(x,-\epsilon_x(y))| &\lesssim \rd_{CC}(x,y)^{-(d+2+m)}\\
|X_{y} K_P(x,-\epsilon_x(y))| &\lesssim \rd_{CC}(x,y)^{-(d+2+m)},
\end{align*}
are true for $x \neq y$ uniformly away from $\partial U$. Here $X_x$ and $X_y$ denote differentiation along $X$ in the $x$-variable and the $y$-variable, respectively. These properties are inherited by $k_P$. 
\end{remark}

Furthermore, by \cite[Proposition 3.1.18]{ponge}, the algebra of $\Psi_H$DOs is invariant under changes of Heisenberg charts, so we may define an algebra of $\Psi_H$DOs on a manifold. We let $\Psi_H^m(M)$ denote the space of Heisenberg operators of order $m$ on the sub-Riemannian $H$-manifold $M$. We summarize the main properties of Heisenberg operators in a theorem \cite{bg,ponge}.

\begin{thm}
\label{l2andwhatnot}
Let $M$ be a sub-Riemannian $H$-manifold. Then $\Psi^*_H(M):=\cup_{m\in \Z} \Psi^m_H(M)$ is a filtered algebra closed under formal adjoint. If $T\in \Psi^0_H(M)$, then $T$ extends to a bounded operator on $L^2(M)$.
\end{thm}

The theory above also applies to operators between smooth vector bundles. Whenever $E\to M$ is a hermitean complex smooth vector bundle, the Serre-Swan theorem (cf. \cite[Theorem 2.10]{ponge}) gives the existence of a projection $p\in C^\infty(M,M_N(\C))$, for some $N$, and a $C^\infty(M)$-linear isomorphism $C^\infty(M,E)\cong pC^\infty(M;\C^N)$. Whenever $E_1,E_2\to M$ are two vector bundles corresponding to $p_1,p_2\in C^\infty(M,M_N(\C))$ we define
\[\Psi^m_H(M;E_1,E_2):=p_2[\Psi^m_H(M)\otimes M_N(\C)] p_1.\]
If $E_1=E_2$ we write $\Psi^m_H(M;E_1)=\Psi^m_H(M;E_1,E_1)$. It is clear that an element $P\in \Psi^m_H(M;E_1,E_2)$ induces a continuous operator $C^\infty(M;E_1)\to C^\infty(M;E_2)$, and if $m=0$ this operator extends to a continuous operator $L^2(M;E_1)\to L^2(M;E_2)$.

\begin{remark}
\label{szegoex}
Examples of Heisenberg operators include the Szeg\"o projections on a contact manifold, see \cite[Chapter III.6]{taylorncom}. Szeg\"o projections are certain orthogonal projections in $\Psi^0_H(M)$ with infinite dimensional range. The existence of a Szeg\"o projection is in general a subtle issue, see \cite[Chapter 4]{bg}. If $M=\partial \Omega$, where $\Omega$ is a strictly pseudo-convex domain in a complex manifold, the orthogonal projection onto the closed subspace
\[H^2(\partial\Omega):=\{f\in L^2(\partial\Omega):\;\exists F\in \mathcal{O}(\Omega) \mbox{    s.t.      } F|_{\partial \Omega}=f\},\]
is a Szeg\"o projection. We denote it by $P_{\partial\Omega}$. The projection $P_{\partial \Omega}$ will be referred to as \emph{the} Szeg\"o projection of $\partial \Omega$. For details, see the discussion in \cite[Section 4.1]{pongecrelle}.
\end{remark}

The prototypical $H$-elliptic differential operators are of sub-Laplacian type. We will make use of an auxiliary operator of sub-Laplacian type whose construction depends on a choice of horizontal vector-fields $(V_j)_{j=1}^N$ which span $H$ at every point of $M$. Consider the differential expression
\[\Delta_H:=\sum_{j=1}^N V_j^*V_j\in \Psi^2_H(M).\]
We define a self-adjoint closure of $\Delta_H$ as follows. Consider the quadratic form
\[\mathfrak{q}_H(f):=\sum_{j=1}^N \int_M |V_jf|^2\rd x,\]
with maximal domain \[\mathrm{Dom}(\mathfrak{q}_H)=\{f\in L^2(M):V_j f\in L^2(M)\forall j\}\ . \] Its domain is the form-closure of $C^\infty(M)$. The form $\mathfrak{q}_H$ is positive, and the unbounded operator $\Delta_H$ is defined to be the associated self-adjoint operator, which is positive since $\mathfrak{q}_H$ is. The resolvent of $\Delta_H$ is compact. For more details, see \cite{bgs}. We let $\lambda_0(\Delta_H) \leq \lambda_1(\Delta_H) \leq \cdots \to \infty$ denote the eigenvalues of $\Delta_H$, counted with multiplicity. By \cite{ponge}, the eigenvalues satisfy a Weyl law: there exists an explicit constant $\nu_0(\Delta_H)$ such that
\begin{equation}
\label{sublapasympt}
\lambda_k(\Delta_H)\sim \nu_0(\Delta_H)k^{\frac{2}{d+2}}.
\end{equation}
We also define a Sobolev scale:
\[W^s_H(M):=(1+\Delta_H)^{-s/2}L^2(M).\]
Since $1+\Delta_H$ is $H$-elliptic, $W^1_H(M)=\mathrm{Dom}(\mathfrak{q}_H)$. For a vector bundle $E\to M$ corresponding to the smooth projection $p$ we write $W^s_H(M;E):=pW^s_H(M;\C^N)$. We note the following consequence of Theorem \ref{l2andwhatnot} and Equation \eqref{sublapasympt}.

\begin{thm}
\label{variousheisenbergthings}
Let $M$ be a sub-Riemannian $H$-manifold, $E_1,E_2\to M$ smooth vector bundles thereon and $D\in \Psi^m_H(M;E_1,E_2)$. The operator $D$ extends to a continuous operator
\[D:W^s_H(M;E_1)\to W^{s-m}_H(M;E_2),\]
which is Fredholm if $D$ is $H$-elliptic. In this case $\ker D\subseteq C^\infty(M;E_1)$. Furthermore, if $E_1=E_2$ and $D$ is $H$-elliptic and self-adjoint, the resolvent of $D$ as an unbounded operator on $L^2(M;E_1)$ satisfies $(D+\lambda)^{-1}\in \mathcal{L}^{(d+2)/m,\infty}(L^2(M;E_1))$.
\end{thm}

\begin{remark}
We note the well known fact that from Proposition \ref{smoothcom} and Theorem \ref{variousheisenbergthings}, both Theorem \ref{mainbdd} and Theorem \ref{mainweakschatten} follow for smooth functions $f$ and $a$, respectively.
\end{remark}

The next proposition is an important observation for the proof of Theorem \ref{mainbdd}.

\begin{prop}
\label{reductionprop}
Assume that $M$ is a closed sub-Riemannian $H$-manifold. Any $D\in \Psi^1_H(M)$ can be written as a finite sum
$$D=\sum_{j=1}^N T_jV_j+R,$$
where $T_j\in \Psi_H^0(M)$, $V_j$ are horizontal vector fields and $R$ is a smoothing operator.
\end{prop}

\begin{proof}
The operator $\Delta_H$ is $H$-elliptic and as such there is an operator $Q\in \Psi^{-2}_H(M)$ such that
$$1-Q\Delta_H, \;1-\Delta_HQ\in \Psi^{-\infty}(M).$$
Any $D\in \Psi_H^1(M)$ can be written as
\[D=DQ\Delta_H+D(1-Q\Delta_H)=\sum_{j=1}^N DQV_j^*V_j+D(1-Q\Delta_H).\]
The assertion follows with $T_j:=DQV_j^*\in \Psi^0_H(M)$ and $R=D(1-Q\Delta_H)$.
\end{proof}

\subsection{The $T1$-theorem of Hyt\"onen-Martikainen}
\label{t1subsection}

The proof of Theorem \ref{mainbdd} will rely on a $T1$-theorem from \cite{hytmart}. The setup is recalled in this section. We restrict our attention to sub-Riemannian $H$-manifolds equipped with its Carnot-Carath\'eodory metric. The results of \cite{hytmart} hold in the larger generality of a geometrically doubling regular quasimetric space. To simplify notation, for a measurable set $U\subseteq M$ we let $|U|$ denote its volume and write $B_{CC}(x,r)$ for the ball in the Carnot-Carath\'eodory metric of radius $r>0$ centered in $x$. If $B=B_{CC}(x,r)$ we let $\kappa B=B_{CC}(x,\kappa r)$ for $\kappa>0$. If $B\subseteq M$ is a ball contained in a Heisenberg chart and $\kappa>0$ is small enough, by Lemma \ref{equivalencesofmetrics} there is an $\epsilon\in (0,\kappa)$ such that $(\kappa-\epsilon) B\subseteq \delta_\kappa^*B\subseteq (\kappa+\epsilon)B$.

\begin{deef}
A function $f \in L^1_{loc}(M)$ belongs to $BMO^p_\kappa(M)$, when there exists a constant $C\geq 0$ such that, for all balls $B \subset M$, there exists an $f_B\in \mathbb{C}$ with
$$\int_B |f-f_B|^p \leq C^p |\kappa B| \ . $$
The seminorm $\|f\|_{BMO^p_\kappa}$ is defined to be the smallest such $C$.
\end{deef}

Following the notation of \cite{hytmart}, we define the function $\lambda(x,r):=|B_{CC}(x,r)|$.

\begin{prop}
\label{volumeasy}
The function $\lambda$ satisfies the asymptotics
$$\lambda(x,r)\sim r^{d+2}\quad\mbox{as}\quad r\to 0.$$
\end{prop}

The proof of the proposition follows from a direct computation in the Koranyi gauge and Lemma \ref{equivalencesofmetrics}. Proposition \ref{volumeasy} shows that the Riemannian measure on $M$ is doubling with respect to the Carnot-Carath\'eodory metric. We denote the diagonal in $M \times M$ by $\Delta$.

\begin{deef}
\label{calzydef}
A function $K : M\times M \setminus \Delta \to \mathbb{C}$ is called a Calder\'{o}n--Zygmund kernel, provided there are $\alpha \in (0,1]$ and $c,C>0$ such that
\begin{align*}
\big|K(x,y)\big| &\leq C \cdot\rd_{CC}(x,y)^{-d-2} \ ,\\
\big|K(x,y)-K(x',y)\big| &\leq C \frac{\rd_{CC}(x,x')^\alpha}{\rd_{CC}(x,y)^{\alpha+d+2}} \quad \forall x,y\mbox{   such that   }\rd_{CC}(x,y)\geq c \rd_{CC}(x,x') \ ,\\
\big|K(x,y)-K(x,y')\big| &\leq C \frac{\rd_{CC}(y,y')^\alpha}{\rd_{CC}(x,y)^{\alpha+d+2}} \quad \forall x,y\mbox{   such that   }\rd_{CC}(x,y)\geq c \rd_{CC}(y,y') \ .
\end{align*}
The smallest such constant $C$ is denoted by $\|K\|_{CZ_\alpha}$.

An operator $f\mapsto Tf$ acting on suitable functions $f$ is called a Calder\'{o}n--Zygmund operator, provided there exists a Calder\'{o}n--Zygmund kernel $K$ such that $$Tf(x) = \int_M K(x,y)\ f(y)\ dy$$
for all $x \not \in \mathrm{supp}\ f$.
\end{deef}

\begin{remark}
\label{CZ1} 
On a sub-Riemannian $H$-manifold, the norm $\|K\|_{CZ_1}$ is bounded by the optimal constant $C$ that satisfies the first estimate $\big|K(x,y)\big| \leq C \cdot\rd_{CC}(x,y)^{-d-2}$ of Definition \ref{calzydef} as well as the two estimates:
$$|X_{x} K(x,y)| \leq C \frac{1}{\rd_{CC}(x,y)^{d+3}} \quad\mbox{and}\quad |X_{y} K(x,y)| \leq C \frac{1}{\rd_{CC}(x,y)^{d+3}}\ ,$$
for any horizontal vector field $X$.
\end{remark}

\begin{deef}
For every ball $B\subset M$ and every $\varepsilon\in (0,1]$, fix a smooth cut-off function $\chi_{B,\varepsilon} \in C^\infty(M)$ with $\chi_B \leq \chi_{B,\varepsilon} \leq \chi_{(1+\varepsilon)B}$. An operator $T$ is called weakly bounded, if there exists $\Lambda>0$ and a function $S : (0,1] \to (0,\infty)$
 with
$\langle T \chi_{B,\varepsilon}, \chi_{B,\varepsilon}\rangle \leq C S(\varepsilon)  |\Lambda B| \ .$
The smallest such $C$ is denoted by $\|T\|_{WBP_{\Lambda, S}}$.
\end{deef}

In our setting, the main theorem of \cite{hytmart} says:

\begin{thm}
\label{thmhytmartikain}
Let $T$ be an $L^2$-bounded integral operator with Calder\'{o}n--Zygmund kernel $K$, $b_1, b_2\in L^\infty(M)$ with $\mathrm{Re}\ b_1, \mathrm{Re}\ b_2 > c > 0$, $\kappa, \Lambda>1$ and $S: (0,1]\to (0,\infty)$. Then $$\|T\|_{\Bo(L^2(M))} \lesssim \|Tb_1\|_{BMO^2_\kappa} + \|T^*b_2\|_{BMO^2_\kappa} + \|b_2 T b_1\|_{WBP_{\Lambda, S}} + \|K\|_{CZ_\alpha} \ .$$
\end{thm}

\large
\section{Commutators of Lipschitz functions with first order $\Psi_H$DOs}
\normalsize
\label{sectiontwo}

In this Section we will prove Theorem \ref{mainbdd}. As in the previous Section, let $M$ denote a closed sub-Riemannian $H$-manifold. It suffices to prove Theorem \ref{mainbdd} for $E=E'=\C$ -- the trivial bundle. We are going to show that the operator $[D,f]$ fits into the framework of Subsection \ref{t1subsection}.

\begin{lem}
\label{bmocontinuity}
Let $T \in \Psi^0_H(M)$. Then $T : L^\infty(M) \to BMO^2(M)$ is continuous.
\end{lem}

\begin{proof}
After a choice of privileged coordinates and a partition of unity, the problem is reduced to a local problem in a chart $U$. It suffices to prove that $T: L^\infty_c(U)\to BMO^2(M)/\C$ is bounded. Fix $x_0\in M$. Take $r>0$. Let $\chi \in C_c^\infty(U)$ be such that $\chi \equiv 1$ on $B_{CC}(x_0, 2r)$, $\chi \equiv 0$ outside $B_{CC}(x_0, 3r)$. For $f \in L^\infty_c(U)$ consider $g_{1,\chi} = T(f \chi)$ and $g_{2,\chi} = T(f (1-\chi))$. As $T$ is continuous on $L^2(U)$,
\begin{align*}
\frac{1}{|B(x_0, r)|}\int_{B(x_0, r)}|g_{1,\chi}|^2 &\leq C \frac{\|\chi f\|_{L^2(U)}^2}{|B_{CC}(x_0, r)|}\leq  C \frac{3^{d+1}|B_{CC}(x_0, r)|}{|B_{CC}(x_0, r)|}\|f\|_{L^\infty(U)}^2    \leq  3^{d+1}C \|f\|_{L^\infty(U)}^2\ .
\end{align*}
If $k_T(x,y)$ is the distributional kernel of $T$ and $T'$ is the operator with integral kernel $k_T(x,y)-k_T(x_0,y)$, then $g_{2,\chi} = T' f\ (1-\chi)$ modulo constants. Indeed, for any $\phi \in C^\infty_c(U)$ with $\int_U \phi = 0$,
\begin{align*}
\langle g_{2,\chi},\phi\rangle & = \langle f(1-\chi),T^*\phi\rangle = \int_U\left[ \mbox{p.v.}\int_U f(y)(1-\chi(y)) k_T(x,y) \phi(x) \rd x\right]\rd y\\
& = \int_U \left[ \mbox{p.v.}\int_U f(y)(1-\chi(y))(k_T(x,y)-k_T(x_0,y))\ \phi(x)\rd x\right]\rd y\\
&= \langle T' f(1-\chi),\phi\rangle  \ .
\end{align*}
This determines the distribution $g_{2,\chi}$ up to a constant. Using Remark \ref{PsiDOprop}, we obtain for a.e.~$x \in B_{CC}(x_0,r)$
\begin{align*}|g_{2,\chi}(x)| &= \big|\mbox{p.v.}\int_U  (k_T(x,y)-k_T(x_0,y)) f(y)(1-\chi(y))\rd y \big|\\
&\leq \int_{\mathrm{supp}(f) \setminus B_{CC}(x_0, 2r)}  \ |k_T(x,y)-k_T(x_0,y)| |f(y)|\rd y\\
&\lesssim \int_{\mathrm{supp}( f) \setminus B_{CC}(x_0, 2r)} \ \frac{\rd_{CC}(x,x_0)}{\rd_{CC}(x,y)^{d+3}} |f(y)|\rd y\\
&\leq \|f\|_{L^\infty(U)}\ \int_{U \setminus B_{CC}(x_0, 2r)} \frac{r}{\rd_{CC}(x,y)^{d+3}}\rd y \lesssim \|f\|_{L^\infty(U)} \ .
\end{align*}
\end{proof}

If $U\subseteq V$ is a precompact open subset of $V\subseteq M$, we write $U\prec \chi\prec V$ if $\chi\in C^\infty_c(V)$ satisfies $0\leq \chi\leq 1$ and $\chi=1$ on $U$. If $V$ lies in a Heisenberg coordinate chart we write $\chi_\epsilon:=\delta_\epsilon^*\chi$. If $U=B_{CC}(x_0,r_1)$ and $V=B_{CC}(x_0,r_2)$ then $U\prec \chi\prec V$ implies $B_{CC}(x_0,(1-\epsilon_0)\epsilon r_1)\prec \chi_\epsilon \prec B_{CC}(x_0,(1+\epsilon_0)\epsilon r_2)$ for some $\epsilon_0>0$ and all small enough $\epsilon>0$.

\begin{lem}
\label{weakbddness}
Assume that $T \in \Psi^0_H(M)$. For $\chi \in C^\infty(M)$ supported in a privileged coordinate chart centered in $u\in M$ with $\chi \equiv 1$ near $u$, there is a $C>0$ such that for $j=1,\ldots ,d$ and any $f \in \textnormal{Lip}_{CC}(M)$,
\begin{equation}
\label{weakequation}
|\langle [T X_j,f]\chi_\varepsilon, \chi_\varepsilon\rangle| \leq C \varepsilon^{-(d+2)}\|f\|_{\textnormal{Lip}_{CC}(M)}\quad\mbox{as}\quad \epsilon\to \infty ,
\end{equation}
where $X_0,X_1,\ldots, X_d$ denotes the frame in the Heisenberg chart.
\end{lem}

We denote the optimal constant in \eqref{weakequation} by $C_{\ref{weakbddness}}$.

\begin{proof}
Again, the problem can be localized in a chart $U$. From two short computations, we have $\|\chi_\varepsilon\|_{L^2(U)}^2 \leq C \varepsilon^{-(d+2)}$ and
$$\langle [T X_j,f]\chi_\varepsilon, \chi_\varepsilon\rangle = \langle [T,f]X_j\chi_\varepsilon, \chi_\varepsilon\rangle+\langle TX_j(f)\chi_\varepsilon, \chi_\varepsilon\rangle\ .$$
The operator $T$ is $L^2$-bounded and $|\langle TX_j(f)\chi_\varepsilon, \chi_\varepsilon\rangle|\leq \|T\|_{\Bo(L^2(M))} \|X_j(f)\|_{L^{\infty}(M)}\|\chi_\varepsilon\|_{L^2(M)}^2$. It therefore suffices to consider the first term.

The kernel of $T$ minus its principal part $T_0$ is the kernel of an operator $\widetilde{T_1} \in \Psi^{-1}_H(U)$. As $\widetilde{T_1}X_j \in \Psi^0_H(U)$ is bounded, we may assume $T = T_0$. Write the Schwartz kernel as $k_T(x,y) = K_T(x,x-y)$.
Further, we have $[T,f]X_j \chi_\epsilon = \epsilon[T,f](X_j\chi)_\epsilon$.
Altogether, it suffices to show the estimate
$$\|[T_0,f](X_j \chi)_\varepsilon\|_{L^\infty(M)} \lesssim \varepsilon^{-1}\|f\|_{\textnormal{Lip}_{CC}(M)}\ .$$
For $\varepsilon>1$, we estimate
\begin{align*}
|([T_0,f](X_j& \chi)_\varepsilon)_{\varepsilon^{-1}}(x)| = \left|\int  \left[f(y) - f_{\varepsilon^{-1}}(x)\right] K_T\left(\frac{x}{\epsilon}, \frac{x-y}{\epsilon}\right) \left[X_j \chi\right](\varepsilon\cdot y)\rd y\right|\\
& = \varepsilon^{-d-2}\left|\int_U \left[f_{\varepsilon^{-1}}(y)-f_{\varepsilon^{-1}}(x)\right]\varepsilon^{d+2} K_T(\varepsilon^{-1}\cdot x, x-y) \left[X_j \chi\right](y)\rd y\right|\\
& \lesssim \varepsilon^{-1}\|f\|_{\textnormal{Lip}_{CC}(M)} \|X_j \chi\|_{L^\infty(M)}\int_{\rd_{CC}(x,y)<1} \rd_{CC}(x,y)^{-d-1}\rd y \lesssim \varepsilon^{-1}\|f\|_{\textnormal{Lip}_{CC}(M)}\ .
\end{align*}
This concludes the proof.
\end{proof}

\begin{lem}
\label{calderonestimates}
For $D\in \Psi^1_H(M)$ and $f\in \textnormal{Lip}_{CC}(M)$, the integral kernel of $[D,f]$ satisfies the Calder\'{o}n-Zygmund estimate:
\[\|[D,f]\|_{CZ}\leq C_{CZ_{1}}(D)\|f\|_{\textnormal{Lip}_{CC}(M)}\ ,\]
for a constant $C_{CZ_1}(D)>0$.
\end{lem}

\begin{proof}
If $k_D(x,y)$ denotes the Schwartz kernel of $D$, the Schwartz kernel of $[D,f]$ is given by $k_D(x,y)(f(y)-f(x))$. From
Remark \ref{PsiDOprop}, for any horizontal vector field $X$ we have
$$|k_D(x,y)| \lesssim \frac{1}{\rd_{CC}(x,y)^{d+2}}\quad\mbox{and} \quad |X_{x} k_D(x,y)| \lesssim \frac{1}{\rd_{CC}(x,y)^{d+3}}.$$
This results in
$$\left|k_D(x,y)(f(y)-f(x))\right| \lesssim \frac{\|f\|_{\textnormal{Lip}_{CC}(M)}}{\rd_{CC}(x,y)^{d+1}}\ .$$
On the other hand, for a horizontal vector field $X$,
$$|X_{x}\!\!\left( k_D(x,y)(f(y)-f(x))\right)| \lesssim \frac{|f(y)-f(x)|}{\rd_{CC}(x,y)^{d+3}} + \frac{|X_{x} f(x)|}{\rd_{CC}(x,y)^{d+2}}\lesssim \frac{\|f\|_{\textnormal{Lip}_{CC}(M)}}{\rd_{CC}(x,y)^{d+2}}$$
almost everywhere. Concerning $|X_{y}\left( k(x,x-y)(f(y)-f(x))\right)|$, the estimate
$$|X_{y} k_D(x,y)| \lesssim \frac{1}{\rd_{CC}(x,y)^{d+3}}$$
similarly allows us to deduce the estimate
$$|X_{y}\left( k_D(x,y)(f(y)-f(x))\right)| \lesssim \frac{\|f\|_{\textnormal{Lip}_{CC}(M)}}{\rd_{CC}(x,y)^{d+2}}$$
almost everywhere. By Remark \ref{CZ1}, these estimates imply the assertion.
\end{proof}

Based on the lemmas above and Theorem \ref{thmhytmartikain}, we can now prove Theorem \ref{mainbdd}. Let us start by considering how Lemma \ref{bmocontinuity} and Lemma \ref{weakbddness} fit into the set up of Subsection \ref{t1subsection}. Take a $T\in \Psi^0_H(M)$ and a horizontal vector field $X_j$. By Lemma \ref{bmocontinuity}, there is a constant $C_{\ref{bmocontinuity}}(T)>0$ only depending on $T$ such that
\[\|[TX_j,f]1\|_{BMO^2(M)}= \|T(X_j(f))\|_{BMO^2(M)}\leq C_{\ref{bmocontinuity}}(T)\|X_j(f)\|_{L^\infty(M)}\leq C_{\ref{bmocontinuity}}(T)|f|_{\textnormal{Lip}_{CC}(M)}.\]
Furthermore, $[TX_j,f]$ is weakly bounded by Lemma \ref{weakbddness}, and in the notation of \cite[Section 2.5]{hytmart},
\[\|[TX_j,f]\|_{WBP_1}\leq C_{\ref{weakbddness}}(T)\|f\|_{\textnormal{Lip}_{CC}(M)}.\]
Here $C_{\ref{weakbddness}}(T)$ denotes the constant of Lemma \ref{weakbddness}.

\begin{proof}
For simplicity assume that $f$ is real-valued. By Proposition \ref{reductionprop} we can assume that $D=\sum_j T_jV_j$ and $D^*=\sum_j T_j'V_j$ for $T_j, T_j'\in \Psi^0_H(M)$. Take a function $\chi\in C^\infty(\R)$ such that $\chi(r)=1$ for $r>1$ and $\chi(r)=0$ for $r<1/2$. Define $T_\epsilon$ as the operator whose integral kernel is
$$\chi\left(\epsilon^{-1}\rd_{CC}(x,y)\right)\left(f(x)-f(y)\right)k_D(x,y),$$
where $k_D$ is the Calder\'{o}n-Zygmund kernel of $D$. For a fixed $\epsilon>0$, $T_\epsilon$ is bounded; even of Hilbert-Schmidt class. Combining the observations above with \cite[Theorem 2.3]{hytmart}, there is a constant $C$ only depending on $d$ such that
\begin{align*}
\|T_\epsilon\|&_{\Bo(L^2(M))} \leq C(\|[D,f]1\|_{BMO^2}+\|[D^*,f]1\|_{BMO^2}+\|[TV_j,f]\|_{WBP_1}+\|[D,f]\|_{CZ})\leq \\
&\leq C\cdot \left(\sum_j \left[C_{\ref{bmocontinuity}}(T_j)+C_{\ref{bmocontinuity}}(T_j')+C_{\ref{weakbddness}}(T_j)+C_{\ref{weakbddness}}(T'_j)\right]+C_{KZ}(D)\right)\|f\|_{\textnormal{Lip}_{CC}(M)}.
\end{align*}
Let $T$ be any strong limit of a strongly convergent subsequence of $T_\epsilon$ as $\epsilon \to 0$ and denote the distributional kernel of $D$ by $\tilde{k}_D$. The kernel of $[D,f]-T$ is of the form $(\tilde{k}_D(x,y) - k_D(x,y))\left(f(x)-f(y)\right)$, where $\tilde{k}_D(x,y) - k_D(x,y)$ is a homogeneous distribution of order $1$ supported in $x=y$. The norm of the associated operator is hence norm bounded by a multiple of $\|f\|_{\textnormal{Lip}_{CC}(M)}$. The theorem follows.
\end{proof}

\large
\section{Commutators of Lipschitz functions with zeroth order $\Psi_H$DOs}
\normalsize
\label{provingweakschattenlipschitz}

In this section we prove that Theorem \ref{mainbdd} implies the first part of Theorem \ref{mainweakschatten}. Namely, we prove that for any $T\in \Psi^0_H(M;E)$ there is a constant $C_{T}>0$ such that any $a\in \textnormal{Lip}_{CC}(M)$ satisfies the estimate
\begin{equation}
\label{tcomma}
 \|[T,a]\|_{\Bo(W^{-1/2}_H(M;E),W^{1/2}_H(M;E))}\leq C_{T}\|a\|_{\textnormal{Lip}_{CC}(M)}.
\end{equation}
The Lipschitz case of Theorem \ref{mainweakschatten} follows from Equation \eqref{sublapasympt} and Equation \eqref{tcomma} above. We will throughout this section assume that $M$ is a closed sub-Riemannian $H$-manifold of dimension $d+1$ and let $\Delta_H\in \Psi^2_H(M)$ denote the sub-Laplacian operator constructed in Subsection \ref{subsectionschatten}. It suffices to consider the case $E=\C^l$ -- a trivial bundle.

\begin{lem}
If there are estimates of the form \eqref{tcomma} for any $l\in \N$ and any $F\in \Psi^0_H(M;\C^l)$ such that
$$F^*-F=F^2-1=0,$$
then there are estimates of the form \eqref{tcomma} for all $l\in \N$ and $T\in \Psi^0_H(M;\C^l)$.
\end{lem}

\begin{proof}
Any element $T\in \Psi^0_H(M;\C^l)$ admits an estimate of the form \eqref{tcomma} if and only if the operator
\[\tilde{T}:=\begin{pmatrix} \lambda& T\\T^*& \lambda\end{pmatrix}\in \Psi^0_H(M;\C^{2l})\]
admits an estimate of the form \eqref{tcomma}. Taking $\lambda>0$ large enough, we can assume that $\tilde{T}$ is an $H$-elliptic operator on $L^2(M;\C^{2l})$, furthermore self-adjoint and strictly positive, i.e. there is an $\epsilon>0$ such that $\tilde{T}-\epsilon$ is positive. Consider the self-adjoint operator
\[F_T:=2(1+\tilde{T}^2)^{-1}\begin{pmatrix} \tilde{T}^2& \tilde{T}\\\tilde{T}& 1\end{pmatrix}-1\in \Psi^0_H(M;\C^{4l}).\]
A direct computation gives 
\[\left((1+\tilde{T}^2)^{-1}\begin{pmatrix} \tilde{T}^2& \tilde{T}\\\tilde{T}& 1\end{pmatrix}\right)^2=(1+\tilde{T}^2)^{-2}\begin{pmatrix} \tilde{T}^4+\tilde{T}^2& \tilde{T}^3+\tilde{T}\\\tilde{T}^3+\tilde{T}& 1+\tilde{T}^2\end{pmatrix}=(1+\tilde{T}^2)^{-1}\begin{pmatrix} \tilde{T}^2& \tilde{T}\\\tilde{T}& 1\end{pmatrix},\]
so $F_T^2=1$.

If an estimate of the form \eqref{tcomma} holds true for any $l\in \N$ and $F\in \Psi^0_H(M;\C^l)$ such that $F^*-F=F^2-1=0$, then it does in particular hold true for $F_T$. However,
\[[F_T,a]=\begin{pmatrix} *&*\\ * & 2[(1+\tilde{T}^2)^{-1},a]\end{pmatrix}=\begin{pmatrix} *&*\\ * & -2(1+\tilde{T}^2)^{-1}[\tilde{T}^2,a](1+\tilde{T}^2)^{-1}\end{pmatrix}.\]
We obtain the estimate
\begin{align*}
 \|[\tilde{T}^2,a]& \|_{\Bo(W^{-1/2}(M;\C^{2l}),W^{1/2}_H(M;\C^{2l}))} \\
&  \leq\|1+\tilde{T}^2\|_{\Bo(W^{1/2}_H,W^{1/2}_H)}\|1+\tilde{T}^2\|_{\Bo(W^{-1/2}_H,W^{-1/2}_H)}\|[F_T,a]\|_{\Bo(W^{-1/2}_H,W^{1/2}_H)}\\
& \leq C_{F_T}\|a\|_{\textnormal{Lip}_{CC}(M)}.
\end{align*}

On the other hand, since $\tilde{T}$ is strictly positive, there is a compact contour $\Gamma$ in the right half plane such that
\[\tilde{T}=\frac{1}{2\pi i}\int_\Gamma \lambda^{1/2}(\lambda-\tilde{T}^2)^{-1}\rd \lambda.\]
Therefore, the commutator $[\tilde{T},a]$ can be expressed as
\[[\tilde{T},a]=\frac{1}{2\pi i}\int_\Gamma \lambda^{1/2}(\lambda-\tilde{T}^2)^{-1}[\tilde{T}^2,a](\lambda-\tilde{T}^2)^{-1}\rd \lambda.\]
Since $\Gamma$ is compact, we arrive at the estimate
$$ \|[\tilde{T},a]\|_{\Bo(W^{-1/2}_H(M;\C^{2l}),W^{1/2}_H(M;\C^{2l}))}\leq C \|[\tilde{T}^2,a]\|_{\Bo(W^{-1/2}_H(M;\C^{2l}),W^{1/2}_H(M;\C^{2l}))}$$
for some $C=C_{T,\Gamma}>0$.
\end{proof}

\begin{lem}
For any $l\in \N$ and $F\in \Psi^0_H(M;\C^l)$ such that $F^2-1=F^*-F=0$ there is an estimate of the form \eqref{tcomma}.
\end{lem}

\begin{proof}
Consider the $H$-elliptic self-adjoint operator
\begin{equation}
\label{dffromt}
D_F:=\begin{pmatrix} 0& \sqrt{1+\Delta_H}F\\ F\sqrt{1+\Delta_H}& 0\end{pmatrix}.
\end{equation}
By \cite[Chapter 5]{ponge}, $D_F\in \Psi^1_H(M;\C^{2l})$. The square of $D_F$ is easily computed to be
\[D_F^2=\begin{pmatrix} 1+\Delta_H&0\\ 0&F(1+\Delta_H)F\end{pmatrix}.\]
Thus, $D_F^2$ is a strictly positive self-adjoint operator, and $|D_F|:=(D_F^2)^{1/2}$ is invertible. We define the bounded operator
\[\tilde{F}:=D_F|D_F|^{-1}=\begin{pmatrix} 0&F\\ F& 0\end{pmatrix}\ .\]

The following operator inequality was proven in \cite[Proposition 1]{sww} for any $a\in \textnormal{Lip}_{CC}(M,i\mathbb{R})$:
\[-\|[D_F,a]\|_{\Bo(L^2(M;\C^{2l}))}|D_F|^{-1}\leq [\tilde{F},a]\leq \|[D_F,a]\|_{\Bo(L^2(M;\C^{2l}))}|D_F|^{-1}.\]
The quadratic form $f\mapsto \langle |D|^{-1}f,f\rangle_{L^2(M;\C^{2l})}$ defines a norm on $W_H^{-1/2}(M;\C^{2l})$, which is equivalent to its usual norm. Hence, an argument using quadratic forms implies that for suitable constants $C_D,C_D'>0$,
\begin{align*}
\| [\tilde{F},a]& \|_{\Bo(W^{-1/2}_H(M;\C^{2l}),W^{1/2}_H(M;\C^{2l}))}\\
&  \leq C_D\|[D_F,a]\|_{\Bo(L^2(M;\C^{2l}))}\leq C_D'\|a\|_{\textnormal{Lip}_{CC}(M)}.
\end{align*}
\end{proof}

\large
\section{Interpolation of Lipschitz and Besov spaces on Heisenberg manifolds}
\normalsize
\label{interpolationsection}

In order to deduce Theorem \ref{mainweakschatten} from the results of Section \ref{provingweakschattenlipschitz} we will interpolate between $C(M)$ and $\textnormal{Lip}_{CC}(M)$. We will do this by means of an anisotropic Besov space scale. The proofs in this section are standard techniques of Besov spaces and interpolation and are included for the convenience of the reader. The main theorem of this section is as follows.

\begin{thm}
\label{interpolatingwithlipschitz}
For a closed sub-Riemannian $H$-manifold,
\[[C(M),\textnormal{Lip}_{CC}(M)]_{\theta,\infty}=C^\theta_{CC}(M),\quad\theta\in (0,1).\]
\end{thm}

\subsection{Local situation}

We first study the local situation, modeled on a given Heisenberg group structure on $\R^{d+1}$ as in Section \ref{Heisenbergsection}. \\

Take a positive function $\varphi\in C^\infty(\R^{d+1})$ such that $\supp(\varphi)\subseteq B_H(0,1)$ (the ball in the Koranyi gauge) and $\int \varphi=1$. We define a sequence $(\varphi_j)_{j\in \N}\subseteq C^\infty_c(\R^{d+1})$ by means of $\varphi_0=\varphi$ and
\[\varphi_j(x)=\varphi(2^{-j}.x).\]
The sequence $(\varphi_j)_{j\in \N}$ satisfies that $\varphi_j\to 1$ pointwise. We also set $\phi_0:=\varphi$ and $\phi_j:=\varphi_j-\varphi_{j-1}$ for $j\geq 1$. It can be verified that $B_H(0,2^{-(j+1)})\subseteq \supp \phi_j\subseteq B_H(0,2^{-j})$ and $\sum_{j=0}^\infty\phi_j=1$ (in pointwise convergence). We define the operators $\Phi_j:=Op(\phi_j)$, i.e. $\Phi_jf:=\check{\phi}_j*f$. Here we use the notation
\[\check{g}(x):=\int g(\xi)\e^{ix\xi}\rd \xi.\]

\begin{prop}
\label{phijalgebra}
The operators $\Phi_j\in \Psi^{-\infty}(\R^{d+1})$ satisfy the identity
$$\Phi_j=(\Phi_{j-1}+\Phi_j+\Phi_{j+1})\Phi_j.$$
\end{prop}

\begin{proof}
Since $\Phi_j$ are Fourier multipliers, this follows from the identity
$$\phi_j=(\phi_{j-1}+\phi_j+\phi_{j+1})\phi_j,$$
which in turn holds true as $\phi_{j-1}+\phi_j+\phi_{j+1}=\varphi_{j+1}-\varphi_{j-1}=1$ on $\supp \phi_j$.
\end{proof}

The estimates in the following proposition will be crucial in relating the Besov scale to the H\"older scale.

\begin{prop}
Let $X$ be a horizontal vector field. The operators $\Phi_j$ satisfy the following estimates
\begin{align}
\label{phijplonkest}
\|\Phi_j f\|_{L^\infty}&\leq 2\|\check{\varphi}\|_{L^1} \|f\|_{L^\infty}\\
\label{phijplonkderivativeest}
\|X\Phi_jf\|_{L^\infty}&\leq  2^jC_\varphi \|f\|_{L^\infty},
\end{align}
where the constant $C_\varphi$ only depends on $\varphi$ and $X$.
\end{prop}

\begin{proof}
A direct computation gives us the identity
\[(\phi_j)\,\check{}\,(x)=2^{j(d+2)}\check{\varphi}(2^j.x)-2^{(j-1)(d+2)}\check{\varphi}(2^{j-1}.x).\]
The estimate $\|(\phi_j)\,\check{}\,\|_{L^1}\leq 2\|\check{\varphi}\|_{L^1}$ follows, proving \eqref{phijplonkest}.

The homogeneity of $X$ guarantees that
$$X(\phi_j)\,\check{}\,(x)=2^{j(d+3)}T_{2^j}X\check{\varphi}-2^{(j-1)(d+3)}T_{2^{j-1}}X\check{\varphi}.$$
We arrive at the estimate $\|X(\phi_j)\,\check{}\,\|_{L^1}\leq C_\varphi 2^j$, where the constant is proportional to $\|X\check{\varphi}\|_{L^1}$. This proves \eqref{phijplonkderivativeest}.
\end{proof}

\begin{deef}[Besov spaces]
For $s\geq0$ and $p,q\in [1,\infty]$ we define
\[\|f\|_{HB^s_{p,q}}:=\left\|\left(2^{sj}\|\Phi_jf(\cdot)\|_{L^p(\R^{d+1})}\right)_{j\in \N}\right\|_{\ell^q(\N)}.\]
The Besov space corresponding to these parameters is defined by
\[HB^s_{p,q}(\R^{d+1}):=\{f:\|f\|_{HB^s_{p,q}}<\infty\}.\]
If $p=q=\infty$, we write
\[C^{s,*}_H(\R^{d+1}):=HB^s_{\infty,\infty}(\R^{d+1}).\]
\end{deef}

The linear spaces $HB^s_{p,q}(\R^{d+1})$ are Banach spaces in the norm $\|\cdot\|_{HB^s_{p,q}}$. We will only use $HB^s_{\infty,p}$ in this paper. Whenever $s\in (0,1)$, we let $C^s_{CC}(\R^{d+1})$ denote the Banach algebra of bounded functions H\"older continuous in the Carnot-Carath\'eodory metric on $\R^{d+1}$. Similarly, we use the notation $C^0_{CC}(\R^{d+1})=BC(\R^{d+1})$ for the Banach algebra of bounded continuous functions and $C^1_{CC}(\R^{d+1})=\textnormal{Lip}_{CC}(\R^{d+1})$ for the Banach algebra of bounded Lipschitz functions.

\begin{lem}
\label{csstatcs}
For $s\in (0,1)$, $C^{s,*}_H(\R^{d+1})=C^{s}_{CC}(\R^{d+1})$ as Banach spaces.
\end{lem}

The proof is along the lines of the proof of \cite[Theorem 6.1]{abels}. We recall it for the convenience of the reader.

\begin{proof}
We prove the inclusion $C^s_{CC}(\R^{d+1})\subseteq C^{s,*}_H(\R^{d+1})$ for $s\in [0,1]$. Assume that $f\in C^s_{CC}(\R^{d+1})$. We have
\[|f(x+y)-f(x)|\leq \|f\|_{C^s_{CC}(\R^{d+1})} |y|_H^s.\]
For $j\geq 1$, we compute
\[\int_{\R^{d+1}}(\phi_j)\,\check{}\,(y)\rd y=\phi_j(0)=0.\]
In particular, we can rewrite $\Phi_j f$ as
\[\Phi_jf(x)=\int_{\R^{d+1}}(\phi_j)\,\check{}\,(y)f(x-y)\rd y=\int_{\R^{d+1}}(\phi_j)\,\check{}\,(y)(f(x-y)-f(x))\rd y.\]
As a consequence, if $s\in [0,1]$
\begin{align}
\sup_{j\in \N}2^{js}\|\Phi_jf\|_{L^\infty}&\leq\sup_{j\in \N}2^{js}\|f\|_{C^s_{CC}}\int_{\R^{d+1}} |y|_H^s|(\phi_j)\,\check{}\,(y)| \rd y\leq\\
\nonumber
&\leq 2\sup_{j\in \N}\|f\|_{C^s_{CC}}\int_{\R^{d+1}} |y|_H^s|\check{\phi}(y)| \rd y\leq C\|f\|_{C^s_{CC}}.
\end{align}

Conversely, assume $f\in C^{s,*}_H(\R^{d+1})$ where $s\in (0,1)$. For $|y|\leq 1$,
\[f(x-y)-f(x)=\sum_{2^j>|y|_H^{-1}}\left[\Phi_jf(x-y)-\Phi_jf(x)\right]+\sum_{2^j\leq |y|_H^{-1}}\left[\Phi_jf(x-y)-\Phi_jf(x)\right].\]
The first term is estimated as follows;
\[\sum_{2^j>|y|_H^{-1}}|\Phi_jf(x-y)-\Phi_jf(x)|\leq 2\sum_{2^j>|y|_H^{-1}}\|\Phi_jf\|_{L^\infty}\leq 2\|f\|_{C^{s,*}_H}\sum_{2^j>|y|_H^{-1}} 2^{-js}\leq C|y|_H^s\|f\|_{C^{s,*}_H}.\]

It is a direct consequence of Proposition \ref{phijalgebra} and  \eqref{phijplonkderivativeest} that
\begin{equation}
\label{phijlipext}
\|\Phi_j f\|_{\textnormal{Lip}_{CC}}= \|(\Phi_{j-1}+\Phi_j+\Phi_{j+1})\Phi_j f\|_{\textnormal{Lip}_{CC}}\leq C 2^j\|\Phi_j f\|_{L^\infty}.
\end{equation}
From these estimates and \eqref{phijplonkderivativeest}, we have
\begin{align*}
\sum_{2^j\leq |y|_H^{-1}}|\Phi_jf(x-y)-\Phi_jf(x)|&\leq \sum_{2^j\leq |y|_H^{-1}}|y|_H\|\Phi_jf\|_{\textnormal{Lip}_{CC}}\leq\\
&\leq  |y|_H\|f\|_{C^{s,*}_H}\sum_{2^j\leq |y|_H^{-1}}2^{j(1-s)}\leq C|y|_H^s\|f\|_{C^{s,*}_H},
\end{align*}
where the last estimate follows from the estimate $\sum_{2^j\leq |y|_H^{-1}}2^{j(1-s)}\leq C|y|_H^{s-1}$.
\end{proof}

\begin{prop}
\label{sandwichinglip}
There are continuous inclusions
$$HB^1_{\infty,1}(\R^{d+1})\hookrightarrow \textnormal{Lip}_{CC}(\R^{d+1})\hookrightarrow C^{1,*}_H(\R^{d+1}).$$
\end{prop}

\begin{proof}
The proof of the fact that the inclusion $\textnormal{Lip}_{CC}(\R^{d+1})\hookrightarrow C^{1,*}_H(\R^{d+1})$ is continuous is carried out in the proof of Lemma \ref{csstatcs}. To prove continuity of the inclusion $HB^1_{\infty,1}(\R^{d+1})\hookrightarrow \textnormal{Lip}_{CC}(\R^{d+1})$, take an $f\in HB^1_{\infty,1}(\R^{d+1})$. The triangle inequality implies the estimate
\[|f(x-y)-f(x)|\leq \sum_{j=0}^\infty|\Phi_jf(x-y)-\Phi_jf(x)|\leq |y|_H\sum_{j=0}^\infty \|\Phi_j f\|_{\textnormal{Lip}_{CC}}.\]
By \eqref{phijlipext},
\[\sum_{j=0}^\infty \|\Phi_j f\|_{\textnormal{Lip}_{CC}}\leq C\sum_{j=0}^\infty 2^j\|\Phi_j f\|_{L^\infty}=C\|f\|_{HB^1_{\infty,1}}.\]

\end{proof}

\begin{prop}
\label{interpolatingbesov}
For any $s\neq t>0$, $\theta\in (0,1)$ and $p\in [1,\infty]$,
\[[HB^s_{\infty,1}(\R^{d+1}),HB^t_{\infty,\infty}(\R^{d+1})]_{\theta,p}=HB^{s_\theta}_{\infty,p}(\R^{d+1})\]
and \[[HB^s_{\infty,1}(\R^{d+1}),HB^s_{\infty,\infty}(\R^{d+1})]_{\theta,p}=HB^{s}_{\infty,p_\theta}(\R^{d+1}).\]
Here $s_\theta=(1-\theta)s+\theta t$, $\frac{1}{p_\theta}=1-\theta$. In particular,
\[[C^{s,*}(\R^{d+1}),C^{t,*}(\R^{d+1})]_{\theta,\infty}=C^{s_\theta,*}(\R^{d+1}),\]
\end{prop}

The proof follows from the techniques in \cite[Chapter 6]{berghlofstrom}.

\begin{prop}
\label{sandwichingc}
There are continuous inclusions
 $$HB^0_{\infty,1}(\R^{d+1})\hookrightarrow BC(\R^{d+1})\hookrightarrow C^{0,*}_H(\R^{d+1}).$$
\end{prop}

\begin{proof}
The second inclusion was shown to be continuous in the proof of Lemma \ref{csstatcs}. To prove the first inclusion, we take an $f\in HB^0_{\infty,1}(\R^{d+1})$ and note
\[\left\|f-\sum_{j=0}^N \Phi_j f\right\|_{L^\infty}\leq \sum_{j=N+1}^\infty \|\Phi_jf\|_{L^\infty}\to 0\quad\mbox{as}\quad N\to \infty.\]
Therefore, $f$ is in the $L^\infty$-closure of $BC^\infty(\R^{d+1})$, i.e.~in $BC(\R^{d+1})$.

\end{proof}

\begin{cor}
\label{interpolatingwithlipschitzrd}
For any $\theta\in (0,1)$,
\[[BC(\R^{d+1}),\textnormal{Lip}_{CC}(\R^{d+1})]_{\theta,\infty}=C^\theta_{CC}(\R^{d+1}).\]
\end{cor}

\begin{proof}
By Proposition \ref{sandwichinglip},
\begin{align*}
HB^1_{\infty,1}(\R^{d+1})=[HB^1_{\infty,1}(\R^{d+1}),&HB^1_{\infty,\infty}(\R^{d+1})]_{1,1}\subseteq \textnormal{Lip}_{CC}(\R^{d+1})\subseteq \\
&\subseteq [HB^1_{\infty,1}(\R^{d+1}),HB^1_{\infty,\infty}(\R^{d+1})]_{1,\infty}=HB^1_{\infty,\infty}(\R^{d+1}).
\end{align*}
By Proposition \ref{sandwichingc},
\begin{align*}
HB^0_{\infty,1}(\R^{d+1})=[HB^0_{\infty,1}(\R^{d+1}),&HB^0_{\infty,\infty}(\R^{d+1})]_{0,1}\subseteq BC(\R^{d+1})\subseteq \\
&\subseteq [HB^0_{\infty,1}(\R^{d+1}),HB^0_{\infty,\infty}(\R^{d+1})]_{0,\infty}=HB^0_{\infty,\infty}(\R^{d+1}).
\end{align*}
As a consequence, for $\theta\in (0,1)$
\[[BC(\R^{d+1}),\textnormal{Lip}_{CC}(\R^{d+1})]_{\theta,\infty}=[HB^0_{\infty,\infty}(\R^{d+1}),HB^1_{\infty,\infty}(\R^{d+1})]_{\theta,\infty}=C^\theta_{CC}(\R^{d+1}),\]
where the last identity follows from Proposition \ref{interpolatingbesov}.

\end{proof}

\subsection{Global situation and proof of Theorem \ref{interpolatingwithlipschitz} }

In this section we will prove Theorem \ref{interpolatingwithlipschitz}. Let $M$ denote a closed sub-Riemannian $H$-manifold. We cover $M$ with Heisenberg charts $(U_i,\kappa_i)_{i=1}^N$. Here we assume that the Heisenberg diffeomorphisms $\kappa_i:U_i\xrightarrow{\sim} V_i\subseteq \R^{d+1}$ extend as diffeomorphisms to slightly bigger domains, preserving the Heisenberg structure. We also take a partition of unity $(\chi_i)_{i=1}^N\subseteq C^\infty(M)$ subordinate to the cover $(U_i)_{i=1}^N$.

\begin{proof}[Proof of Theorem \ref{interpolatingwithlipschitz}]
The choice of Heisenberg charts and partition of unity induces the following factorization of the identity operator on $C(M)$:
\[C(M)\to \bigoplus_{i=1}^N BC(\R^{d+1})\to C(M).\]
Similarly, we factor the identity operator on $\textnormal{Lip}_{CC}(M)$:
\[\textnormal{Lip}_{CC}(M)\to \bigoplus_{i=1}^N \textnormal{Lip}_{CC}(\R^{d+1})\to \textnormal{Lip}_{CC}(M).\]
By Corollary \ref{interpolatingwithlipschitzrd}, we can factor the identity operator on $[C(M),\textnormal{Lip}_{CC}(M)]_{\theta,\infty}$ as
\[[C(M),\textnormal{Lip}_{CC}(M)]_{\theta,\infty}\to \bigoplus_{i=1}^N C_{CC}^{\theta}(\R^{d+1})\to [C(M),\textnormal{Lip}_{CC}(M)]_{\theta,\infty}.\]
The theorem follows.
\end{proof}

\subsection{Proof of Theorem \ref{mainweakschatten}}

We saw in Section \ref{provingweakschattenlipschitz} that Theorem \ref{mainweakschatten} holds true in the case of Lipschitz functions. We will now proceed to prove the general case using interpolation.

\begin{prop}
\label{interpolatingwithcompactsforealthistime}
For any separable Hilbert space $\He$ and $p\in (1,\infty)$,
\[[\Ko(\He),\ellL^{p,\infty}(\He)]_{\theta,\infty}=\ellL^{p/\theta,\infty}(\He).\]
\end{prop}

\begin{proof}
Recall, for instance from \cite[Appendix B, Chapter IV]{c}, for any separable Hilbert space $\He$
\[[\Ko(\He),\ellL^1(\He)]_{\theta,\infty}=\ellL^{1/\theta,\infty}(\He).\]
The proposition follows from this and the Re-iteration Theorem, see \cite[Section 3.5]{berghlofstrom}.
\end{proof}

\begin{proof}[Proof of Theorem \ref{mainweakschatten}]
Let $M$ be a closed sub-Riemannian $H$-manifold of dimension $d+1$, $E\to M$ a vector bundle and $T\in \Psi^0_H(M;E)$.
By Theorem \ref{variousheisenbergthings}, the inclusions $W^{1/2}_H(M,E)\to L^2(M,E)$ and $L^2(M,E)\to W^{-1/2}_H(M,E)$ induce a continuous injection
$$ \Bo(W^{-1/2}_H(M;E),W^{1/2}_H(M;E))\to \mathcal{L}^{d+2,\infty}(L^2(M;E)).$$
The results of Section \ref{provingweakschattenlipschitz} imply the fact that $\mathfrak{C}_T:\textnormal{Lip}_{CC}(M)\to\mathcal{L}^{d+2,\infty}(L^2(M;E))$, $a\mapsto [T,a]$ is continuous. Using the fact that $\mathfrak{C}_T:C(M)\to \Ko(L^2(M))$, $a\mapsto [T,a]$ is continuous, Theorem \ref{mainweakschatten} follows from Theorem \ref{interpolatingwithlipschitz} and Proposition \ref{interpolatingwithcompactsforealthistime}.
\end{proof}

In the following proposition we improve the weak Schatten norm estimate of Theorem \ref{mainweakschatten} for symbols $a\in [C(M),\textnormal{Lip}_{CC}(M)]_\alpha\subseteq C^\alpha_{CC}(M)$ -- the slightly smaller complex interpolation space. The inclusion $[C(M),\textnormal{Lip}_{CC}(M)]_\alpha\subseteq C^\alpha_{CC}(M)$ is strict if $\alpha\in (0,1)$, since $\textnormal{Lip}_{CC}(M)$ is dense in $[C(M),\textnormal{Lip}_{CC}(M)]_\alpha$ by \cite[Theorem 4.2.2]{berghlofstrom}.

\begin{prop}
\label{complexinterandsobreg}
Assume that $T\in \Psi^0_H(M;E)$. There is a constant $C_{T}>0$ such that for any $a\in [C(M),\textnormal{Lip}_{CC}(M)]_\alpha$,
\[ \|[T,a]\|_{W^{-\alpha/2}_H(M,E)\to W^{\alpha/2}_H(M;E)}\leq 2^{1-\alpha}\|T\|_{\Bo(L^2(M))}^{1-\alpha}C_{T}^\alpha\|a\|_{[C(M),\textnormal{Lip}_{CC}(M)]_\alpha}.\]
\end{prop}

\begin{proof}
As above, we can assume that $E$ is the trivial line bundle. We let $D:=\sqrt{\Delta_H+1}$, so $W^\alpha_H(M)=\mathrm{Dom}(D^\alpha)$. By \cite[Theorem 3]{seeley71}, for any Banach space $X$, operator $L$ with bounded imaginary powers on $X$ and $\alpha\in [0,1]$,
\[[X,\mathrm{Dom}(L)]_\alpha=\mathrm{Dom}(L^\alpha).\]
We let $X=\Bo(L^2(M))$ and $L(T)=D^{1/2}TD^{1/2}$ with $\Dom(L)=\Bo(W^{-1/2}_H(M),W^{1/2}_H(M))$. The operator $L$ has bounded imaginary powers because $L^z(T)=D^{z/2}TD^{z/2}$ and $D$, being self-adjoint, has bounded imaginary powers. Since 
$$\mathrm{Dom}(L^\alpha)=\Bo(W^{-\alpha/2}_H(M),W^{-\alpha/2}_H(M)),$$ 
we have
\[ [\Bo(L^2(M)),\Bo(W^{-1/2}_H(M),W^{1/2}_H(M))]_\alpha=\Bo(W^{-\alpha/2}_H(M),W_H^{\alpha/2}(M)).\]

 By the results of Section \ref{provingweakschattenlipschitz}, $\tilde{\mathfrak{C}}_T:\textnormal{Lip}_{CC}(M)\to\Bo(W^{-1/2}_H(M),W^{1/2}_H(M))$, $a\mapsto [T,a]$ is continuous. We estimate
\[\|[T,a]\|_{\Bo(L^2(M))}\leq 2\|T\|_{\Bo(L^2(M))}\|a\|_{\Bo(L^2(M))}=2\|a\|_{C(M)},\]
so $\tilde{\mathfrak{C}}_T:C(M)\to \Bo(L^2(M))$, $a\mapsto [T,a]$ is continuous. By interpolation, the considerations above imply the fact that $\tilde{\mathfrak{C}}_T$ defines a continuous mapping
$$ [C(M),\textnormal{Lip}_{CC}(M)]_\alpha\to \Bo(W^{-\alpha/2}_H(M),W^{\alpha/2}_H(M)),$$
with norm bounded by $2^{1-\alpha}\|T\|_{\Bo(L^2(M))}^{1-\alpha}C_{T}^\alpha$.
\end{proof}

\large
\section{Applications}
\normalsize
\label{applicationssection}

\subsection{The Carnot-Carath\'eodory metric as a limit point of Connes metrics}
\label{ccmetsubse}

A well known result in noncommutative geometry asserts that a spectral triple on a unital $C^*$-algebra determines a metric on its state space. It is a natural question to ask whether the Carnot-Carath\'eodory metric comes from a spectral triple of the same metric dimension as the Hausdorff dimension of the Carnot-Carath\'eodory metric. Partial answers were given in \cite{hasselmann}. A natural place to look for a spectral triple is in the Heisenberg calculus, as Ponge's results \cite{ponge} on spectral asymptotics give the correct metric dimension. The metric dimension of a unital spectral triple $(\mathcal{A},\He,D)$ is defined as the number
\[\inf \,\{p\in \R_+: (i+D)^{-1}\in \ellL^p(\He)\}.\]
On a sub-Riemannian manifold a horizontal Dirac operator does not give rise to a spectral triple (see \cite{hasselmann}). In this section we will show that the Carnot-Carath\'eodory metric may be obtained as the Gromov-Hausdorff limit of a family of metrics coming from a spectral triple whose metric dimension coincides with the Hausdorff dimension. Recent partial results were obtained in \cite{hasselmann}. We recall the definition of Gromov-Hausdorff distance.

\begin{deef}[Gromov-Hausdorff distance]
Let $(X,\rd_X)$ and $(Y,\rd_Y)$ be two compact metric spaces. The Gromov-Hausdorff distance between $(X,\rd_X)$ and $(Y,\rd_Y)$ is the infimum of all $r\geq 0$ such that there exist isometric embeddings $i:X\to Z$ and $j:Y\to Z$ into a compact metric space $(Z,\rd_Z)$ with $\rd_Z(i(X),j(Y))\leq r$.
\end{deef}

A notion related to spectral triples, better adapted to sub-Riemannian metrics, was introduced in \cite{hasselmann}: that of a \emph{degenerate} spectral triple. A degenerate spectral triple $(\mathcal{A},\He,D)$ satisfies the usual axioms except that $(D+i)^{-1}$ is compact only on the orthogonal complement of all infinite-dimensional eigenspaces.

\begin{deef}[The Connes metric]
Associated with a (degenerate) spectral triple there is a metric $\rd_D$ on the state space $S(A)$ defined by
\[\rd_D(\omega,\omega'):=\sup\{|\omega(a)-\omega'(a)|:a\in \mathcal{A}:\; \|[D,a]\|_{\Bo(\He)}\leq 1\}.\]
\end{deef}

Connes metrics are well studied, see for instance \cite[Chapter VI]{c} or \cite{kerrli,rieffel1,rieffel2}.
The prototypical example is when $M$ is a closed Riemannian manifold, $\mathcal{A}=C^\infty(M)$, $\He=L^2(M;S)$ for a Clifford bundle $S\to M$ and $D$ is a Dirac type operator on $S$. In this case, $\rd_D$ coincides with the geodesic distance on $M$.

In \cite{hasselmann}, a degenerate spectral triple was constructed on any Riemannian Carnot manifold from a choice of a spin structure $S^HM$ (see Definition \ref{cantotdef} on page \pageref{cantotdef}) on the horizontal bundle $H$ (assuming $H$ admits a spin structure). This degenerate spectral triple takes the form $(C^\infty(M),L^2(M,S^HM),D^H)$, where $D^H$ is a horizontal Dirac operator on $S^HM$. A key feature of this degenerate spectral triple is the identity
\[[D^H,f]=c_H(\rd^Hf),\]
where $c_H:H^*\to \mathrm{End}(S^HM)$ denotes Clifford multiplication and $\rd^H:C^\infty(M)\to C^\infty(M,H^*)$ denotes exterior horizontal differentiation. In particular, it follows from the definition of Connes metric that 
\[\rd_{D^H}=\rd_{CC},\]
see \cite[ Theorem 3.3.6]{hasselmann}. We recall the contents of \cite[Proposition $7.3.1.$ii)]{hasselmann}. In \cite{hasselmann}, the result is stated in larger generality on Carnot manifolds.

\begin{prop}
Let $M$ be a closed sub-Riemannian $H$-manifold equipped with a horizontal Dirac operator $D^H:C^\infty(M,S^HM)\to C^\infty(M,S^HM)$. If $(0,1]\ni \theta\mapsto D_\theta\in \Psi^1_H(M,S^HM )$ is a family of operators such that
\[\forall \epsilon>0\;\exists \delta>0:\quad \quad \|[D_\theta-D^H,f]\|_{\Bo(L^2(M,S^HM))}<\epsilon\|f\|_{\textnormal{Lip}_{CC}(M)}\;\forall \theta\in (0,\delta),\]
then for any $\epsilon>0$ there is a $\delta>0$ such that
\[(1-\epsilon)\rd_{D_\theta}(x,y)\leq \rd_{CC}(x,y)\leq (1+\epsilon)\rd_{D_\theta}(x,y)\quad\forall x,y\in M,\; \theta\in (0,\delta).\]
\end{prop}

In particular, $\rd_{D_\theta}\to \rd_{CC}$ with respect to Gromov-Hausdorff distance. We arrive at the following corollary of Theorem \ref{mainbdd}.

\begin{cor}
\label{ghlimit}
Let $M$ be a $d+1$-dimensional closed sub-Riemannian $H$-manifold. Assume that $S\in \Psi^1_H(M,S^HM)$ is an operator such that
\[D_\theta:=D^H+\theta S\quad\mbox{is $H$-elliptic for all small enough} \quad\theta>0.\]
Then $(C^\infty(M),L^2(M,S^HM),D_\theta)$ is a spectral triple of metric dimension $d+2$ for all small enough $\theta>0$ satisfying that with respect to Gromov-Hausdorff distance 
\[(M,\rd_{D_\theta})\to (M,\rd_{CC}),\quad\mbox{as}\quad \theta\to 0.\]
\end{cor}

\begin{remark}
The operator $S\in \Psi^1_H(M,S^HM)$ can be chosen arbitrarily as long as the principal symbol of $S$ acts as the projection onto the kernel of the principal symbol of $D^H$ in any fiberwise representation of the symbol algebra.
\end{remark}

\subsection{A regularized functional on the H\"older functions}
\label{subsregfunc}

This subsection studies a regularized variant of a multilinear functional from complex analysis on spaces of H\"{o}lder continuous functions. For smooth functions on $S^{2n-1}$, it goes back to \cite{engguozh}. The results were later extended to pre-compact strictly pseudo-convex domains  $\Omega\subseteq \C^n$ with smooth boundary in \cite{engzh}. We abuse the terminology and also include $S^1$ whenever we refer to a contact manifold, with its usual Szeg\"o projection and the scale of $\rd_{CC}$-H\"older spaces given by $C^\alpha_{CC}(S^1)=C^{\alpha/2}(S^1)$ for $\alpha\in (0,2]$.\newline

First, we need to recall some facts on Dixmier traces. The reader is referred to \cite{sukolord} for details. A state $\omega\in \ell^\infty(\N)^*$ is called a generalized limit if
\begin{enumerate}
\item For any $x\in c_0(\N)$, $\omega(x)=0$.
\item If $x=(x_0,x_1,x_2,\ldots)$ and
$$\sigma_n(x)=(\underbrace{x_0,x_0,\ldots,x_0}_{n \;\mathrm{times}},\underbrace{x_1,x_1,\ldots,x_1}_{n \;\mathrm{times}},\underbrace{x_2,x_2,\ldots,x_2}_{n \;\mathrm{times}},\ldots),$$
then $\omega(x)=\omega(\sigma_n(x))$.
\end{enumerate}
Let $\He$ be a separable Hilbert space. Any generalized limit $\omega$ induces a continuous trace $\tra_\omega:\mathcal{L}^{1,\infty}(\He)\to \C$ by the formula 
\[\tra_\omega(T)=\omega\left(\frac{\sum_{k=0}^N \mu_k(T)}{\log(N+2)}\right)_{N\in \N}.\]
Here $0\leq T\in \mathcal{L}^{1,\infty}(\He)$ has singular values $(\mu_k(T))_{k\in \N}$.
See more in \cite[Theorem 1.3.1]{sukolord}. Dixmier traces can often be computed in smooth settings. For instance, a well known theorem proved independently by Connes and Wodzicki states: if $T$ is a pseudo-differential operator of order $-n$ on an $n$-dimensional closed manifold $M$, then
\[\tra_\omega(T)=\frac{1}{n(2\pi)^n}\int_{S^*M} \sigma_{-n}(T)\ \rd S.\]
Here $\rd S$ denotes the surface measure. If $T$ acts on sections of a vector bundle, the same formula holds with $\sigma_{-n}(T)$ replaced by the fiberwise trace of $\sigma_{-n}(T)$. A similar identity is known in the Heisenberg calculus by results of Ponge \cite{pongeresidue}. 

\begin{deef}
Let $\omega$ be a generalized limit, $M$ a $2n-1$-dimensional contact manifold and $P_M\in \Psi^0_H(M)$ a projection. For $k\geq n$ we define the functionals
\[\xi_{k,\omega},\zeta_{k,\omega}:\underbrace{C^{\frac{n}{k}}_{CC}(M)\hat{\otimes}C^{\frac{n}{k}}_{CC}(M)\hat{\otimes}\cdots \hat{\otimes}C^{\frac{n}{k}}_{CC}(M)\hat{\otimes} C^{\frac{n}{k}}_{CC}(M)}_{2k \;\mathrm{times}}\to \C,\]
by the formulas
\begin{align}
\xi_{k,\omega}(a_1\otimes \cdots \otimes a_{2k}):=\tra_\omega\left(P_Ma_1(1-P_M)a_2P_M\cdots P_Ma_{2k-1}(1-P_M)a_{2k}P_M\right),\\
\zeta_{k,\omega}(a_1\otimes \cdots \otimes a_{2k}):=\tra_\omega\left([P_Ma_1P_M,P_Ma_2P_M]\cdots [P_Ma_{2k-1}P_M,P_Ma_{2k}P_M]\right).
\end{align}
\end{deef}
For $k=n$, recall our convention $C^1_{CC}(M)=\textnormal{Lip}_{CC}(M)$.

In the case $M=\partial \Omega$, for a pre-compact strictly pseudo-convex domain $\Omega\subseteq \C^n$ with smooth boundary, and $P$ being the Szeg\"o projection, the multilinear functional $\zeta_{n,\omega}$ was studied in \cite{engzh} for smooth functions. By \cite[Theorem 11]{engzh}, if $a_1,\ldots, a_{2n}\in C^\infty(\partial\Omega)$, we have
\begin{equation}
\small
\label{engszhangcomp}
\zeta_{n,\omega}(a_1\otimes \cdots \otimes a_{2n})=\frac{1}{n!(2\pi)^n} \int_{\partial \Omega} \mathcal{L}^*(\bar{\partial}_ba_1,\bar{\partial}_ba_2)\cdots \mathcal{L}^*(\bar{\partial}_ba_{2n-1},\bar{\partial}_ba_{2n}) \eta\wedge (\rd\eta)^{n-1}.
\normalsize
\end{equation}
Here $\bar{\partial}_b$ denotes the boundary $\bar{\partial}$-operator, $\mathcal{L}^*$ the dual Levi form of $\partial \Omega$ (cf. Remark \ref{leviformremark}) and $\eta$ the contact form on $\partial \Omega$. The following argument shows $\zeta_{n,\omega}\neq 0$. Take a contact diffeomorphism $\phi:U\to U'\subseteq S^{2n-1}$ of a small neighborhood $U\subseteq \partial \Omega$ onto a neighborhood $U'$ of $(0,0,\ldots,1)\in S^{n-1}$. For a real--valued function $\chi\in C^\infty_c(U)$ of sufficiently small support, we consider the functions $a_{2j-1}:=\chi\cdot \phi^*z_1$ and $a_{2j}:=\chi\cdot \phi^*\bar{z}_1$, for $j=1,\ldots,n$. It follows from a simple computation and \cite[Equation (3)]{engzh} that $\zeta_{n,\omega}(a_1\otimes \cdots \otimes a_{2n})>0$ if $U'$ is small enough.

\begin{remark}
Let us make two computational remarks about $\xi_{k,\omega}$ and $\zeta_{k,\omega}$. If $M=\partial\Omega$ and $a_1,a_3,\ldots, a_{2k-1}\in C^{\frac{n}{k}}_{CC}(\partial \Omega)\cap \mathcal{O}(\Omega)$ and $\bar{a}_2,\bar{a}_4,\ldots, \bar{a}_{2k}\in C^{\frac{n}{k}}_{CC}(\partial \Omega)\cap \mathcal{O}(\Omega)$ then
$$P_{\partial \Omega}a_{2j-1}(1-P_{\partial \Omega})a_{2j}P_{\partial \Omega}=[P_{\partial \Omega}a_{2j-1}P_{\partial \Omega},P_{\partial \Omega}a_{2j}P_{\partial \Omega}],$$
so in this case $\xi_{k,\omega}(a_1\otimes \cdots \otimes a_{2k})=\zeta_{k,\omega}(a_1\otimes \cdots \otimes a_{2k})$.

In general, we observe
\begin{align}
\label{xirefo}
P_Ma_1(1-P_M)a_2P_M&\cdots P_Ma_{2k-1}(1-P_M)a_{2k}P_M\\
\nonumber
&=P_M[P_M,a_1][P_M,a_2]\cdots [P_M,a_{2k-1}][P_M,a_{2k}]\ ,\\
\label{zetarefo}
[P_Ma_1P_M,P_Ma_2P_M]&\cdots [P_Ma_{2k-1}P_M,P_Ma_{2k}P_M]\\
\nonumber
&=P_M\left[[P_M,a_1],[P_M,a_2]\right]\cdots \left[[P_M,a_{2k-1}],[P_M,a_{2k}]\right]
\end{align}
\end{remark}

Let $V_k\subseteq (C^{\frac{n}{k}}_{CC}(M))^{\hat{\otimes}2k}$ be the closed subspace spanned by elements of the form $a_1\otimes \cdots \otimes a_{2k}$ such that there exists a $j$  and a $\beta>n/k$ with $a_j\in C^\beta_{CC}(M)$.

\begin{thm}
\label{vanishingcor}
The multilinear functionals $\xi_{k,\omega}$ and $\zeta_{k,\omega}$ are continuous and do not vanish in general. Furthermore,
$$\xi_{k,\omega}|_{V_k}=\zeta_{k,\omega}|_{V_k}=0.$$
\end{thm}

\begin{proof}
The computations in Equation \eqref{xirefo} and Equation \eqref{zetarefo}, together with Theorem \ref{mainweakschatten}, show the fact that $\xi_{k,\omega}$ and $\zeta_{k,\omega}$ are continuous. If $a_j\in C^\beta_{CC}(M)$ for a $\beta>n/k$, then the operators
\begin{align*}
P_M[P_M,a_1][P_M,a_2]\cdots [P_M,a_{2k-1}][P_M,a_{2k}],&\\
P_M\left[[P_M,a_1],[P_M,a_2]\right]\cdots &\left[[P_M,a_{2k-1}],[P_M,a_{2k}]\right]
\end{align*}
belong to $ \mathcal{L}^1(L^2(M))$, so their Dixmier traces vanish. It follows from Equation \eqref{engszhangcomp} and the discussion thereafter that $\zeta_{n,\omega}$ does not vanish. The fact that $\xi_{k,\omega}$ does not vanish in general follows from Lemma \ref{nonvanlem} below.
\end{proof}

\begin{remark}
The continuous Hochschild $2k$-cocycle on $C_{CC}^{\frac{n}{k}}(M)$ given by
$$\tilde{\xi}_{k,\omega}(a_0\otimes a_1\otimes \cdots \otimes a_{2k}):=\tra_\omega(P_M a_0[P_M,a_1]\cdots [P,a_{2k}])$$
satisfies $\xi_{k,\omega}(a_1\otimes \cdots \otimes a_{2k})=\tilde{\xi}_{k,\omega}(1\otimes a_1\otimes \cdots \otimes a_{2k})$. In the notation of \cite{connesncdg} one can therefore write $\xi_{k,\omega}=B_0\tilde{\xi}_{k,\omega}$. The functionals $\xi_{k,\omega}$ and $\zeta_{k,\omega}$ are related by
\small
$$\zeta_{k,\omega}(a_1\otimes \cdots a_{2k})=\sum_{(l_1,\ldots, l_k)\in \{0,1\}^k} (-1)^{\sum_{j=1}^k l_j} \xi_{k,\omega}(a_{1+l_1}\otimes a_{2-l_1}\otimes a_{3+l_2}\otimes a_{4-l_2}\otimes \cdots \otimes a_{2k-1+l_k}\otimes a_{2k-l_k}).$$
\normalsize
\end{remark}

If $M=S^1$ with its contact structure as the boundary of the unit disc, $C^\alpha_{CC}(S^1)=C^{\alpha/2}(S^1)$. In particular, $\xi_{k,\omega}$ is a multilinear functional on $C^{\frac{1}{2k}}(S^1)$.

\begin{lem}
\label{nonvanlem}
Define $W\in C^{1/4}(S^1,\R)$ as
$$W(z):=\sum_{k=0}^\infty 2^{-n/4}(z^{2^n}+z^{-2^n}).$$
Then $\xi_{2,\omega}(W,W,W,W)\geq \sqrt{6}/\log(2)$ for any generalized limit $\omega$.
\end{lem}

\begin{proof}
By \cite[Theorem $4.9$, Chapter II]{zygbook}, $W\in C^{1/4}(S^1,\R)$. For simplicity we denote the Szeg\"o projection on $S^1$ by $P$. Let $e_l(z)=z^l$ for $l\in \Z$ be the standard Fourier-basis for $L^2(S^1)$, the eigenbasis associated with $z\partial_z\in \Psi^2_H(S^1)=\Psi^1(S^1)$. $P$ is characterized by its action n these basis vectors: $Pe_l=e_l$ if $l\geq 0$ and $Pe_l=0$ if $l<0$. For $l\in \Z$
\begin{align*}
\langle PW(1-P)WPW(1-P)WPe_l,e_l\rangle&=\langle P[P,W]WPe_l,[P,W]We_l\rangle\\
&=
\begin{cases}
\|P[P,W]We_l\|_{L^2(S^1)}^2,\;&l\geq 0,\\
0,\;&l<0.
\end{cases}
\end{align*}
The operator
$$PW(1-P)WPW(1-P)WP=(PW(1-P)WP)^2=(P[P,W]WP)^2$$
is a positive operator. We let $(\lambda_k(PW(1-P)WP)^2)_{k\in \N}$ denote its eigenvalues. A consequence of the min-max principle and the positivity of $\omega$ is the estimate
\begin{align*}
\tra_\omega(PW(1-P)WPW(1-P)WP)&=\omega\left(\frac{\sum_{k=0}^N\lambda_k(PW(1-P)WP)^2}{\log N}\right)_{N\geq 1}\\
&\geq \omega\left(\frac{\sum_{k=0}^N\|P[P,W]WPe_k\|^2_{L^2(S^1)}}{\log N}\right)_{N\geq 1}.
\end{align*}
Hence, the lemma follows once we prove the following lower estimate:
\begin{equation}
\label{estimatingthatthang}
\|P[P,W]We_l\|_{L^2(S^1)}^2\geq \frac{2}{\log(2)l}.
\end{equation}
We define $g_l:=[P,W]We_l$. The Binomial Theorem implies the identity
\begin{align*}
g_l(z)&=\frac{1}{2\pi i}\int_0^{2\pi} \frac{W(\e^{i\theta})-W(z)}{\e^{i\theta}-z}W(\e^{i\theta})\e^{il\theta}\rd \theta\\
&=\frac{1}{2\pi i}\int_0^{2\pi}\sum_{n,m=0}^\infty \sum_{k=0}^{2^n-1} 2^{-(n+m)/4}\left(\e^{i(2^n-1-k+l+2^m)\theta}+\e^{i(2^n-1-k+l-2^m)\theta}\right)z^k\rd \theta\\
&\quad-\frac{1}{2\pi i}\int_0^{2\pi}\sum_{n,m=0}^\infty \sum_{k=0}^{2^n-1} 2^{-(n+m)/4}\left(\e^{i(k+l-2^n+2^m)\theta}+\e^{i(k+l-2^n-2^m)\theta}\right)z^{-k-1}\rd \theta\\
&=\sum_{n=0}^\infty \sum_{2^n-l\geq 2^m\geq l+1}2^{-(n+m)/4}z^{2^n-2^m+l}-\sum_{n=0}^\infty \sum_{l-2^n\leq 2^m\leq l-1}2^{-(n+m)/4}z^{l-2^n-2^m-1}.
\end{align*}
In particular,
$$Pg_l(z)=\sum_{n=0}^\infty \sum_{2^n-l\geq 2^m\geq l+1}2^{-(n+m)/4}z^{2^n-2^m+l},$$
and we arrive at the norm computation 
\begin{equation}
\label{computationofpgl}
\|Pg_l\|_{L^2(S^1)}^2=\sum_{(n',n,m',m)\in \Gamma_l} 2^{-(n+n'+m+m')/4},
\end{equation}
where $\Gamma_l$ denotes the set of $(n',n,m',m)\in \N^4$ such that
\[2^n-l\geq 2^m\geq l+1,\quad 2^{n'}-l\geq 2^{m'}\geq l+1\quad\mbox{and}\quad 2^n-2^m=2^{n'}-2^{m'}.\]
We estimate
\[\|Pg_l\|_{L^2(S^1)}^2\geq \sum_{2^n\geq 2l+1}2^{-n/2}(2^n-l)^{-1/2}\geq \int_{\log_2(2l+2)}^\infty 2^{-x/2}(2^x-l)^{-1/2}\rd x.\]
After Taylor expanding $2^{-x/2}(2^x-l)^{-1/2}=2^{-x}(1-2^{-x}l)^{-1/2}$, we arrive at the identity
\[2^{-x/2}(2^x-l)^{-1/2}=\sum_{n=0}^\infty \begin{pmatrix} -1/2\\n\end{pmatrix}(-l)^n2^{-(n+1)x}.\]
The right hand side converges absolutely for $x\geq \log_2(2l+1)$. As such,
\begin{align*}
\int_{\log_2(2l+2)}^\infty 2^{-x/2}(2^x-l)^{-1/2}\rd x&=\frac{1}{\log(2)}\sum_{n=0}^\infty \begin{pmatrix} -1/2\\n\end{pmatrix}\frac{(-1)^{n+1}}{n+1}\frac{l^n}{(l+1)^{n+1}}\frac{1}{2^{n+1}}\\
&=\frac{1}{l\log(2)}\sum_{n=0}^\infty \begin{pmatrix} -1/2\\n\end{pmatrix}\frac{(-1)^{n+1}}{n+1}\left(\frac{l}{2l+2}\right)^{n+1}.
\end{align*}
Again using a Taylor series, for $(1+l/(2l+2))^{1/2}$ we arrive at the identity
$$\int_{\log_2(2l+2)}^\infty 2^{-x/2}(2^x-l)^{-1/2}\rd x=\frac{2}{l\log(2)}\left(1+\frac{l}{2l+2}\right)^{1/2}.$$
Since $(1+l/(2l+2))^{1/2}\geq \sqrt{3/2}$, Equation \eqref{estimatingthatthang} follows.
\end{proof}

\begin{remark}
Combining Lemma \ref{nonvanlem} with Theorem \ref{vanishingcor}, we deduce that the continuous functional
$$C^{1/4}(S^1)\ni f\mapsto \xi_{2,\omega}(f,W,W,W)$$
is non-vanishing, but vanishes on the closure of $C^\beta(S^1)$ for any $\beta>1/4$.
\end{remark}

\begin{lem}
\label{zetaandxitheorem}
If $(e_l)_{l\in \N}\subseteq C^\infty(M)$ forms an eigenbasis of $L^2(M)$ for a self-adjoint $H$-elliptic operator $D\in \Psi^m_H(M)$, then for any $a_1,a_2,\ldots, a_{2n}\in \textnormal{Lip}_{CC}(M)$:
\begin{align}
\label{xicomp}
\xi_{n,\omega}&(a_1\otimes \cdots \otimes a_{2n})\\
\nonumber
&=\omega\left[\left(\frac{\sum_{l=0}^N\langle P_Ma_1(1-P_M)a_2P_M\cdots P_Ma_{2n-1}(1-P_M)a_{2n}P_Me_l,e_l\rangle}{\log(N+2)}\right)_{N\in \N}\right],\\
\label{zetacomp}
\zeta_{n,\omega}&(a_1\otimes \cdots \otimes a_{2n})\\
\nonumber
&=\omega\left[\left(\frac{\sum_{l=0}^N\langle [P_Ma_1P_M,P_Ma_2P_M]\cdots [P_Ma_{2n-1}P_M,P_Ma_{2n}P_M]e_l,e_l\rangle}{\log(N+2)}\right)_{N\in \N}\right].
\end{align}
\end{lem}

\begin{proof}
To prove the identities \eqref{xicomp} and \eqref{zetacomp}, we may assume that $D$ is strictly positive. Indeed, we otherwise replace $D$ by $D^2+1$, if $D$ is of positive order, and with $(D+R)^{-2}+1$, if $D$ is of negative order, for a suitable self-adjoint smoothing operator $R$. For simplicity we also assume that the order of $D$ is $1/2$. By \eqref{xirefo} and \eqref{zetarefo} the identities \eqref{xicomp} and \eqref{zetacomp} are a consequence of the following: If $T\in \mathcal{L}^{1,\infty}(L^2(M))$ is of the form  $T=T'T_0T''$, where $T', T''$ extend to continuous operators $T',T'':W^{-1/2}_H(M)\to W^{1/2}_H(M)$ and $T'T_0, T_0T'' \in \mathcal{L}^{2n/(2n-1),\infty}(L^2(M))$, then
\begin{equation}
\label{tramod}
\tra_\omega(T)=\omega\left(\frac{\sum_{l=0}^N\langle Te_l,e_l\rangle}{\log(N+2)}\right)_{N\in \N}.
\end{equation}
We show this as follows: If $T'$ extends to a continuous operator $T':W^{-1/2}_H(M)\to W^{1/2}_H(M)$, the operator $DT'D$ extends to a bounded operator on $L^2(M)$. Similarly for $T''$, if $T''$ extends to $T'':W^{-1/2}_H(M)\to W^{1/2}_H(M)$ then $DT''D$ extends to a bounded operator. By the Weyl law in the Heisenberg calculus, $D^{-1}\in \mathcal{L}^{(4n,\infty)}(L^2(M))$. One arrives at $T'T_0 D^{-1}\in \mathcal{L}^{4n/(4n-1),\infty}(L^2(M))$ and 
\[ TD=(T'T_0 D^{-1}) DT''D\in \mathcal{L}^{4n/(4n-1),\infty}(L^2(M)).\]
Similarly, $D^{-1}T_0T''\in \mathcal{L}^{4n/(4n-1),\infty}(L^2(M))$ and
$$ DT=DT'D(D^{-1}T_0T'')\in \mathcal{L}^{4n/(4n-1),\infty}(L^2(M)).$$
Equation \eqref{tramod} follows from \cite[Lemma 11.2.10]{sukolord} using the decomposition
$$ \tra_\omega(T)=\tra_\omega (\mathrm{Re} \,T)+i \tra_\omega(\mathrm{Im}\, T).$$
\end{proof}

\begin{remark}
Both sides of Equation \eqref{engszhangcomp} define continuous multilinear functionals on $\textnormal{Lip}_{CC}(\partial\Omega)$. It remains open whether \eqref{engszhangcomp} in fact holds true for all functions that are Lipschitz in the Carnot-Carath\'eodory metric. Lemma \ref{zetaandxitheorem} indicates that to prove Equation \eqref{engszhangcomp} in general, it suffices to prove continuity of $\zeta_{n,\omega}$ in a very weak topology.
\end{remark}

\begin{remark}
It is unclear if the functionals $\xi_{k,\omega}$ and $\zeta_{k,\omega}$ are computable for $k>n$. A first approach would be to prove the natural analogue of Lemma \ref{zetaandxitheorem}. The method of proof requires regularity in the Sobolev scale $W^s_H$, which we only know for the complex interpolation space $[C(M),\textnormal{Lip}_{CC}(M)]_\alpha$. The analogue of Lemma \ref{zetaandxitheorem} for $\zeta_{k,\omega}$ and $\xi_{k,\omega}$ is trivially true provided $a_j\in [C(M),\textnormal{Lip}_{CC}(M)]_{n/k}$ for some $j$. In this case both sides vanish by Theorem \ref{vanishingcor}.

As another obstacle $\xi_{k,\omega}$ and $\zeta_{k,\omega}$ generally depend on the generalized limit $\omega$ \cite{gg}. A more elementary problem would be to compute the exact value of $\xi_{2,\omega}(W,W,W,W)$ using Equation \eqref{computationofpgl}.
\end{remark}

\section*{{\bf Acknowledgements}}
This work was supported by the Danish National Research Foundation (DNRF) through the Centre for Symmetry and Deformation and the Danish Science Foundation (FNU) through research grant 10-082866. The authors also thank the mathematics institutes of Leibniz University of Hannover and the GRK1463 for facilitating this collaboration. The authors gratefully acknowledge the support by the Institut Mittag-Leffler (Djursholm, Sweden) where this work was initiated. The second author wishes to thank Stefan Hasselmann for discussions on the Carnot-Carath\'eodory metric. Thanks to Fedor Sukochev and Steven Lord for pointing out the result  \cite[Lemma 11.2.10]{sukolord} of their book providing us with a proof of the computational Lemma \ref{zetaandxitheorem}.


\begin{thebibliography}{100}

\bibitem{abels} H. Abels, \emph{Pseudodifferential and singular integral operators - an introduction with applications}, De Gruyter graduate lectures (2012).

\bibitem{alroz} A. Alexandrov, G. Rozenblum, \emph{Finite rank Toeplitz operators: some extensions of D. L\"ucking's theorem}, J. Funct. Anal. 256 (2009), no. 7, p. 2291--2303.

\bibitem{alexpeller} A. Aleksandrov, V. V. Peller, \emph{Functions of operators under perturbations of class ${\bf S}_p$}, J. Funct. Anal. 258 (2010), no. 11, p. 3675--3724.

\bibitem{baajjulg} S. Baaj, P. Julg, \emph{Theorie bivariante de Kasparov et operateurs non bornes dans les $C^*$-modules hilbertiens}, C. R. Acad. Sci. Paris Ser. I Math. 296 (1983), no. 21, 875--878.


\bibitem{blackadar} B. Blackadar, \emph{$K$-theory for Operator Algebras}, Cambridge 1998.

\bibitem{bgs} R. Beals, P. Greiner, N. K. Stanton, \emph{The heat equation on a $CR$-manifold}, J. Differential Geom. 20 (1984), no. 2, p. 343--387.

\bibitem{bg}R.~Beals, P.~Greiner, \emph{Calculus on Heisenberg manifolds}, Ann.~of Math.~Studies, vol.~119, Princeton Univ.~Press, Princeton, N.J.~1988.

\bibitem{bella96} A. Bellaiche, \emph{The tangent space in sub-Riemannian geometry}, Sub-Riemannian geometry, p. 1--78,
Progr. Math., 144, Birkh\"auser, Basel, 1996.

\bibitem{berghlofstrom} J. Bergh, J. L\"ofstr\"om, \emph{Interpolation Spaces}, Springer, Berlin -- Heidelberg -- New York (1976).

\bibitem{birsolinterpolation} M. S. Birman, M. Z. Solomyak, \emph{Application of interpolational methods to estimates of the spectrum of integral operators} (Russian) Operator theory in function spaces (Proc. School, Novosibirsk, 1975) (Russian), p. 42--70, 341, "Nauka'' Sibirsk. Otdel., Novosibirsk, 1977.

\bibitem{calderonone} A. P. Calder\'{o}n, \emph{Commutators of singular integral operators}, Proc. NAS, USA 53 (1965), p. 1092--1099.

\bibitem{cadapaaaow} L. Capogna, D. Danielli, S.D. Pauls, J.T. Tyson, \emph{An introduction to the Heisenberg group and the sub-Riemannian isoperimetric problem}, Birkh\"auser, 2007.

\bibitem{christ} M. Christ, D. Geller, P. Glowacki, L. Polin, \emph{Pseudo-differential operators on groups with dilations}, Duke Mathematical Journal 68 no. 1, p. 31--65, (1992).

\bibitem{ceemtwo} R. Coifman, Y. Meyer, \emph{Commutateurs d'integrales singulieres et operateurs multilineaires}, Ann. Inst. Fourier Grenoble 28 (1978), p. 177--202.

\bibitem{connesncdg} A. Connes, \emph{Noncommutative differential geometry}, Inst. Hautes Etudes Sci. Publ. Math. No. 62 (1985), p. 257--360.

\bibitem{c}A.~Connes, \emph{Noncommutative Geometry}, Academic Press, London 1994.


\bibitem{chow} W.-L. Chow, \emph{\"Uber Systeme von linearen partiellen Differentialgleichungen erster Ordnung}, Mathematische Annalen 117 (1939), p. 98--105.

\bibitem{engguozh} M. Englis, K. Guo, G. Zhang, \emph{Toeplitz and Hankel operators and Dixmier traces on the unit ball of $\C^n$}, Proc. Amer. Math. Soc. 137 (2009), no. 11, p. 3669--3678.

\bibitem{engzh} M. Englis, G. Zhang, \emph{Hankel operators and the Dixmier trace on strictly pseudoconvex domains}, Doc. Math. 15 (2010), p. 601--622.

\bibitem{fangxia} Q. Fang, J. Xia, \emph{Schatten class membership of Hankel operators on the unit sphere}, J. Funct. Anal. 257 (2009), no. 10, p. 3082--3134.

\bibitem{fangxia2} Q. Fang, J. Xia, \emph{A local inequality for Hankel operators on the sphere and its application}, J. Funct. Anal. 266 (2014), p. 876--930.

\bibitem{fangxia3} Q. Fang, J. Xia, \emph{On the membership of Hankel operators in a class of Lorentz ideals}, J. Funct. Anal. 267 (2014), p. 1137--1187.

\bibitem{gg} H. Gimperlein, M. Goffeng, \emph{Nonclassical spectral asymptotics and Dixmier traces of Hankel operators}, preprint.

\bibitem{goffodd} M. Goffeng, \emph{Analytic formulas for topological degree of non-smooth mappings: the odd-dimensional case}, Advances in Mathematics, Volume 231, Issue 1 (2012).

\bibitem{goffeven} M. Goffeng, \emph{Analytic formulas for topological degree of non-smooth mappings: the even-dimensional case}, Journal of Pseudo-Differential Operators and Applications: Volume 4, Issue 2 (2013).

\bibitem{goffmesland} M. Goffeng, B. Mesland, \emph{Finite summability and the noncommutative geometry of Cuntz-Krieger algebras},  Documenta Mathematica 20 (2015).

\bibitem{gromov} M. Gromov, \emph{Carnot-Carath\'eodory spaces seen from within}, Sub-Riemannian geometry, p. 79--323,
Progr. Math., 144, Birkh\"auser, Basel, 1996.

\bibitem{muellerbesov} Y. Han, D. M\"uller, D. Yang, \emph{A theory of Besov and Triebel-Lizorkin spaces on metric measure spaces modeled on Carnot-Carath\'eodory spaces}. Abstr. Appl. Anal. 2008, Art. ID 893409, 250 pp.

\bibitem{hasselmann} S. Hasselmann, \emph{Spectral Triples on Carnot Manifolds}, PhD-thesis, Hannover (2013).

\bibitem{higsonroe} N. Higson, J. Roe, \emph{Analytic K-homology}, Oxford Mathematical Monographs. Oxford Science Publications. Oxford University Press, Oxford, 2000. xviii+405 pp.

\bibitem{hytmart} T. Hyt\"onen, H. Martikainen, \emph{Non-homogeneous Tb theorem and random dyadic cubes on metric measure spaces}, J. Geom. Anal. 22 (2012), no. 4, p. 1071--1107.

\bibitem{kerrli} D.~Kerr, H.~Li, \emph{On Gromov-Hausdorff convergence for operator metric spaces}, J.~Operator Theory 62 (2009), p. 83--109.

\bibitem{keonene} K. Koenig, \emph{On Maximal Sobolev and H\"older estimates for tangential Cauchy-Riemann operator and
Boundary Laplacian}, Amer. J. Math. 124 (2002), p. 129--197.

\bibitem{jeto} K. Knudsen-Jensen, K. Thomsen, \emph{Elements of $KK$-theory}, Birkh�user 1991.

\bibitem{liruss} S.-Y. Li, B. Russo, \emph{Hankel operators in the Dixmier class}, C. R. Acad. Sci. Paris Ser. I Math. 325 (1997), no. 1, p. 21--26.

\bibitem{sukolord} S. Lord, F. Sukochev, D. Zanin, \emph{Singular traces. Theory and applications.} de Gruyter Studies in Mathematics, 46. De Gruyter, Berlin, (2013).

\bibitem{melin} A. Melin, \emph{Lie filtrations and pseudodifferential operators}, preprint (1982).

\bibitem{montgomery} R. Montgomery, \emph{A tour of sub-Riemannian geometries, their geodesics and applications}, American Mathematical Society (2002).

\bibitem{nspap} A. Nagel, E. Stein, \emph{The $\Box_b$-heat equation on pseudoconvex manifolds of finite type in $\C^2$}, Math. Z.
238 (2001), p. 37--88.

\bibitem{nswpap} A. Nagel, E. Stein, S. Wainger, \emph{Balls and metrics defined by vector fields. I. Basic properties}, Acta
Math. 155 (1985), p. 103--147.

\bibitem{peller} V. V. Peller, \emph{Hankel operators and their applications}, Springer Monographs in Mathematics. Springer-Verlag, New York, 2003. xvi+784 pp. ISBN: 0-387-95548-8

\bibitem{peloso} M. M. Peloso, \emph{Hankel operators on weighted Bergman spaces on strongly pseudoconvex domains},
Illinois J. Math. 38 (1994), no. 2, p. 223--249.

\bibitem{pongeresidue} R. Ponge, \emph{Noncommutative residue for Heisenberg manifolds. Applications in CR and contact geometry}, J. Funct. Anal. 252 (2007), no. 2, p. 399--463.

\bibitem{ponge} R. Ponge, \emph{Heisenberg calculus and spectral theory of hypoelliptic operators on Heisenberg manifolds}, American Mathematical Society (2008).

\bibitem{pongecrelle} R. Ponge, \emph{Noncommutative residue invariants for $CR$ and contact manifolds}, J. Reine Angew. Math.
614 (2008) p. 117--151.

\bibitem{rieffel1}M.~Rieffel, \emph{Metrics on state spaces}, Doc.~Math.~4 (1999), p. 559--600.

\bibitem{rieffel2}M.~Rieffel, \emph{Gromov-Hausdorff distance for quantum metric spaces},
Mem.~Amer.~Math.~Soc.~168 (2004), no.~796, p. 1--65.


\bibitem{groz} G. Rozenblum, \emph{Finite rank Toeplitz operators in the Bergman space}, Around the research of Vladimir Maz'ya. III, p. 331--358, Int. Math. Ser. (N.Y.), 13, Springer, New York, 2010.

\bibitem{rozshi} G. Rozenblum, N. Shirokov, \emph{Finite rank Bergman-Toeplitz and Bargmann-Toeplitz operators in many dimensions}, Complex Anal. Oper. Theory 4 (2010), no. 4, p. 767--775.

\bibitem{russo} B. Russo, \emph{On the Hausdorff-Young Theorem for Integral Operators}, Pacific Journal of Mathematics,
Vol 68, No. 1, 197.

\bibitem{seeley71} R. Seeley, \emph{Norms and domains of the complex powers $A_B^z$}, Amer. J. Math. 93 (1971), p. 299--309.

\bibitem{sww} E. Schrohe, M. Walze, J. - M. Warzecha, \emph{Construction de triplets spectraux a partir de modules de Fredholm}, C. R. Acad. Sci. Paris 326 (1998), p. 1195--1199, math/9805063.

\bibitem{simon} B. Simon, \emph{Trace ideals and their applications}, volume 120 of Mathematical Surveys and Monographs, American Mathematical Society, Providence, RI, second edition, (2005).

\bibitem{taylorncom} M. E. Taylor, \emph{Noncommutative microlocal analysis. I.} Mem. Amer. Math. Soc. 52 (1984), no. 313, iv+182 pp.

\bibitem{nlin} M. E. Taylor, \emph{Pseudodifferential operators and nonlinear PDE}, Progress in Mathematics, 100. Birkh\"auser Boston, Inc., Boston, MA, 1991. 213 pp. ISBN: 0-8176-3595-5.

\bibitem{zygbook} A. Zygmund, \emph{Trigonometric series. Vol. I}, Third edition. Cambridge Mathematical Library. Cambridge University Press, Cambridge (2002).

\end{thebibliography}
\end{document}